\documentclass[ ]{article}
\usepackage[calc]{picture}
\usepackage{graphicx}
\usepackage{subfigure}
\usepackage{graphics}
\usepackage{xcolor,tikz}
\usepackage{dsfont}
\usetikzlibrary{decorations.pathreplacing}
\usepackage{amssymb,amsmath,amsthm,bm,setspace,mathrsfs}
\usepackage{marvosym}
\usepackage{url,color,units}
\usepackage[T1]{fontenc}

\DeclareMathOperator{\Dim}{Dim_{\rm H}}
\DeclareMathOperator{\card}{card}
\DeclareMathOperator{\Recurrent}{\sim_{{\rm (R)}}}
\DeclareMathOperator{\nRecurrent}{{\not\sim}_{{\rm (R)}}}

\newcommand{\N}{\mathbf{N}}
\newcommand{\Z}{\mathbf{Z}}

\newcommand{\R}{\mathbf{R}}

\newcommand{\cS}{\mathcal{S}}
\newcommand{\cV}{\mathcal{V}}
\newcommand{\cN}{\mathcal{N}}
\def\D{\mathcal{D}}

\RequirePackage{color}
\RequirePackage[colorlinks, urlcolor=my-blue,linkcolor=my-link,citecolor=my-green]{hyperref}

\definecolor{my-link}{rgb}{0.5,0.0,0.0}
\definecolor{my-blue}{rgb}{0.0,0.0,0.6}
\definecolor{my-red}{rgb}{0.5,0.0,0.0}
\definecolor{my-green}{rgb}{0.0,0.5,0.0}
\definecolor{darkgreen}{rgb}{0.2,0.45,0}
\definecolor{nicosred}{rgb}{0.65,0.0,0.0}
\definecolor{nicos-red}{rgb}{0.0,0.6,0.7}
\definecolor{light-gray}{gray}{0.6}
\definecolor{really-light-gray}{gray}{0.8}

\newcommand{\be}{\begin{equation}}
\newcommand{\ee}{\end{equation}}

\renewcommand{\P}{\mathrm{P}}
\newcommand{\E}{\mathrm{E}}

\newcommand{\1}{\mathds{1}}
\renewcommand{\d}{{\rm d}}

\newcommand{\e}{{\rm e}}
\renewcommand{\geq}{\geqslant}
\renewcommand{\leq}{\leqslant}
\renewcommand{\ge}{\geqslant}
\renewcommand{\le}{\leqslant}

\author{Nicos Georgiou\\University of Sussex
\and Davar Khoshnevisan\\University of Utah
\and Kunwoo Kim\\University of Utah
\and Alex D. Ramos\\Federal University of Pernambuco}
\title{\bf The dimension of the range of a transient random walk\thanks{
	Research supported in part by the NSF grants DMS-1006903 and DMS-1307470}}
\date{October 30, 2014}
\newtheorem{stat}{Statement}[section]
\newtheorem{proposition}[stat]{Proposition}
\newtheorem{corollary}[stat]{Corollary}
\newtheorem{theorem}[stat]{Theorem}
\newtheorem{lemma}[stat]{Lemma}
\theoremstyle{definition} 
\newtheorem{definition}[stat]{Definition}
\newtheorem{remark}[stat]{Remark}
\newtheorem{CONJ}{Conjecture}
\newtheorem{OP}{Problem}
\newtheorem{example}[stat]{Example}

\numberwithin{equation}{section}

\begin{document}
\maketitle
\begin{abstract} 
We find formulas for the macroscopic Minkowski and Hausdorff dimensions of
the range of an arbitrary transient walk in $\Z^d$. 
This endeavor solves a problem of Barlow and Taylor (1991).\\


\end{abstract}

\section{ Introduction}

Let $X:=\{X_n\}_{n=0}^\infty$ 
denote a random walk on $\Z^d$,
started at a nonrandom point $X_0:=a\in\Z^d$, and denote its range by
$\mathcal{R}_X:=\cup_{n=0}^\infty\{X_n\}$.
Our goal is to give an answer to the following question of Barlow and Taylor
\cite[Problem, p.\ 145]{BT1992}:
\begin{equation}\label{BT:question}
		\text{What is $\Dim(\mathcal{R}_X)$?}
\end{equation}
Here,  $\Dim(G)$ denotes the macroscopic Hausdorff dimension
of a set $G\subset\R^d$, as was defined in \cite{BT1989,BT1992}.
We will recall the formal definition and first properties of $\Dim$ in 
\S\ref{subsec:Dim} below. In informal terms, $\Dim(G)$
measures the  ``local dimension of $G$ at infinity,'' and describes the 
large-scale geometry of $G$ in a manner that is similar
to the way that the ordinary, microscopic, 
Hausdorff dimension of $G$ describes
the small-scale geometry of $G$.

In the case that $X$ is a recurrent random walk on $\Z^d$,
one can answer \eqref{BT:question} easily. Recall that a point
$x\in\Z^d$ is \emph{possible} for $X$ if $P^0\{ X_n=x\}>0$
for some $n\ge 1$, where $P^a$ denotes the conditional
law of $X$ given that $X_0=a$, as usual. 
The collection of all possible points of $X$ is an additive
subgroup of $\Z^d$, whence homomorphic to $\Z^k$
for some integer $0\le k \le d$, thanks to
the structure theory of finitely-generated abelian groups.
Therefore, the strong Markov property implies that
$\mathcal{R}_X$ is homomorphic
to $\Z^k$ a.s., and hence we have $\Dim(\mathcal{R}_X)=\Dim(\Z^k)=k$
a.s.\  \cite{BT1992}. 
In other words, \eqref{BT:question}
has non-trivial content if and only if $X$ is transient. This explains
why Barlow and
Taylor \cite{BT1992}  pose \eqref{BT:question} only for transient random walks.

As far as we know, the only positive result 
about \eqref{BT:question} is due to
Barlow and Taylor  themselves \cite[Cor.\ 7.9]{BT1992}. In order to
describe their result, let $g$ denote the Green function of $X$.
That is,
\begin{equation}\label{eq:g}
	g(a\,,x) := \sum_{n=0}^\infty P^a\{X_n=x\},\hskip0.4in x,a\in\Z^d.
\end{equation}
Of course, $g(a\,,x)=g(0\,,x-a)$ as well, and the transience of $X$
is equivalent to the finiteness of the function $g$ on all of
$\Z^d\times\Z^d$. 

Question \eqref{BT:question} was in part motivated by the following
positive result.

\begin{proposition}[Barlow and Taylor \protect{%
	\cite[Corollary 7.9]{BT1992}}]\label{pr:BT} Let $d\geq 2$.
	Suppose there exist  constants $\alpha\in(0\,,2]$
	and $A,B\in(0\,,\infty)$ such that
	\begin{equation}\label{cond:BT}
		A\|x\|^{-d+\alpha}\le 
		g(0\,,x) \le B\|x\|^{-d+\alpha},
	\end{equation}
	whenever $x\in\Z^d\setminus\{0\}$. Then, 
	\begin{equation}
		\Dim(\mathcal{R}_X)=\alpha
		\qquad P^0\text{-a.s.}
	\end{equation}
\end{proposition}

A principal goal of this article is to answer  question \eqref{BT:question} in
general. We do this by following a suggestion of Barlow and Taylor and
introducing a random-walk ``index'' that is equal to 
$\Dim(\mathcal{R}_X)$. 
Moreover, as is tacitly implied in \cite{BT1992},
that index is defined solely in
terms of the statistical properties of $X$.
That index turns out to be related  to
the notion of a  ``recurrent set'' for $X$.

Recall that a set $F\subseteq\Z^d$ is said to be 
\emph{recurrent} for $X$, under $P^a$,
if the random set $X^{-1}(F):=\{n\ge 0: X_n\in F\}$ is unbounded
$P^a$-a.s. This definition makes sense regardless of whether $F$ is random
or not.

Because $X$ is transient,
a necessary condition for the recurrence of $F$ is
that $F$ is unbounded, though this condition is not sufficient,
as the following example shows:
Let $X$ denote the simple symmetric walk
on $\Z^3$, and define $F:=\cup_{k=1}^\infty\{x_k\}$,
where $x_k:=(0\,,0\,,k^3)$.
By the local central limit theorem,
$P^0\{x_k\in\mathcal{R}_X\}=O(k^{-3})$ as $k\to\infty$, hence
\begin{equation}
	E^0 \left[ \sum_{n=0}^\infty\1_F(X_n) \right]=\sum_{k=0}^\infty
	P^0\{x_k\in\mathcal{R}_X\}<\infty,
\end{equation}
thanks to Tonelli's theorem.
Therefore, $F$ is not recurrent for $X$, though it is unbounded.
By $E^0$ we mean the expectation operator for $P^0.$

A necessary-and-sufficient condition for the recurrence
of a \emph{nonrandom} set $F$ was  found first by It\^o and McKean \cite{IM}
in the case that $X$ is the simple random walk on $\Z^d$
for some $d\ge 3$. Lamperti \cite{Lam1963}
discovered a necessary-and-sufficient condition 
in the case that $X$ belongs to a large family of transient random walks on $\Z^d$,
and in fact many transient Markov chains on countable
state spaces. When $X$ is a general transient random walk on $\Z^d$, or
more generally a Markov chain on a countable state space,
there is also an exact condition for set recurrence, but that condition is more involved;
see Bucy \cite{Bucy1965} and, more recently, Benjamini et al \cite{BPP1995}.

All latter works involve various notions of abstract capacity that are
borrowed from probabilistic potential theory. Our answer to \eqref{BT:question}
is also stated in terms of a sort of abstract capacity condition;
see Corollary \ref{cor:main}. One can prove that when the underlying
random walk $X$ is sufficiently regular---for example as in the statement of
Proposition \ref{pr:BT}---our condition can simplify. 
See Corollary \ref{cor:RW:Dim} for  instance.
Still, our index is very difficult to work with in general; see \S\ref{sec:conj}
for a discussion of this shortcoming, and for potential remedies.

We conclude the Introduction with an outline of the paper.
In \S\ref{sec:Backgd} we include some of the technical prerequisites
to reading this paper. Then, in \S\ref{sec:fp}
we  develop a macroscopic theory of ``fractal percolation''
that is the large-scale analogue of  the microscopic theory of
fractal percolation \cite{Mandelbrot,Peres1996}. 
Our macroscopic extension of the microscopic theory is not entirely trivial,
but will  ring familiar to many experts. 
In \S\ref{sec:forest} we write a forest representation
of $\Z^d$ and use it together with the theory of two-parameter processes 
\cite{Kh2002}
in order to characterize exactly when $\mathcal{R}_X$ intersects
a piece of a macroscopic fractal percolation set. 
This is the truly novel part of the present article,
and is likely to have other uses particularly in computing the ordinary and/or
large-scale dimension of complex random sets.
Finally, in \S\ref{sec:mc} we adapt an elegant replica method
of Peres \cite{Peres1996} 
to the present setting in order to compute
the dimension of $\mathcal{R}_X\cap F$ for every recurrent nonrandom set $F$.
In the section that follows this discussion [\S\ref{sec:Dim_M}] we
derive the following simpler and more elegant representation for the macroscopic
Minkowski dimension of an arbitrary random walk on $\Z^d$:
\begin{equation}
	\inf\left\{
	\gamma\in(0\,,d):\ \sum_{x\in\Z^d\setminus\{0\}} \frac{g(0\,,x)}{
	\|x\|^\gamma}<\infty\right\}
\end{equation}
In the last section \S\ref{sec:conj} we state a few remaining open
problems and related conjectures.

\section{ Background material}\label{sec:Backgd}
This section introduces the prerequisite material, necessary for later use.

\subsection{Macroscopic Hausdorff dimension}\label{subsec:Dim}
Throughout we follow the original notation of Barlow and Taylor
\cite{BT1989,BT1992} by setting
\begin{equation}\label{B:S}
	\cV_k:= [-2^k,\,2^k)^d,\quad
	\cS_0:=\cV_0,\quad
	\cS_{k+1} := \cV_{k+1}\setminus\cV_k,
\end{equation}
for all $k\ge 0$.
For every integer $n\in\Z$ define $\D_n$ to be
  the collection of all dyadic cubes $Q^{(n)}$ of the form
\begin{equation}\label{cube}
	Q^{(n)}:=[ j_12^n\,,(j_1+1)2^n)\times\cdots\times
	[j_d2^n\,,(j_d+1)2^n),
\end{equation}
where $j_1,\ldots,j_d\in\Z$
are integers.

If a cube $Q^{(n)}$ has the form \eqref{cube}, then
we say that $j:=(j_1,\ldots,j_d)$ is the \emph{south-west corner}
of $Q^{(n)}$, and $2^n$ is the \emph{sidelength} of $Q^{(n)}$.

By $\mathcal{D}$ we mean the collection of
all dyadic hypercubes of $\Z^d$; that is,
\begin{equation}
	\D:=\bigcup_{n=-\infty}^\infty\D_n.
\end{equation}
A special role is played by
\begin{equation}
	\D_{\ge 1}:=\bigcup_{n=0}^\infty
	\D_n.
\end{equation}
This is the collection of all dyadic cubes of sidelength
no smaller than $ 1$.

For every $\alpha \in (0\,,\infty)$ and  $A \subseteq \Z^d$ define  
\begin{equation}\label{eq:new:h}
	\mathcal{N}_{\alpha}(A\,, \cS_k) := 
	\min_{A \cap \cS_k \subseteq \bigcup_{i=1}^m Q_i} 
	\sum_{i=1}^{m} 2^{\alpha(\ell_i -k-1)},
\end{equation}
where the minimum is taken over all possible coverings of 
$A \cap \cS_k$ by cubes $Q_i$ of sidelength 
$2^{\ell_i} \ge 1$ and south-west corner
$2^{\ell_i} x_i \in 2^{\ell_i}\Z^d$ . All these cubes are elements of $\D_{\ge 1}$.

The \textit{macroscopic Hausdorff dimension} \cite{BT1989,BT1992}
of $A$ is defined by\footnote{%
	Barlow and Taylor
	\cite{BT1989,BT1992} wrote $\tilde{\nu}_\alpha$
	in place of our  $\mathcal{N}_\alpha$, and
	$\dim_{\rm H}$ in place of our $\Dim$. We prefer $\mathcal{N}_\alpha$
	as it reminds us that our $\mathcal{N}_\alpha$ is the large-scale analogue of
	Besicovitch's $\alpha$-dimensional net measures. And we use
	for $\Dim$ in favor of $\dim_{\rm H}$ to distinguish between large-scale and ordinary
	Hausdorff dimension.
	}
\begin{equation}\label{eq:Haus:dim}
		\Dim(A) := \inf\bigg\{ \alpha>0: \sum_{k=1}^\infty
		\cN_{\alpha}(A\,,\cS_k)< \infty\bigg\}.
\end{equation}

It is easy to see that if $A\subseteq B$, then
\begin{equation}
	0\le \Dim(A)\le\Dim(B)\le d\qquad
\end{equation}
The second seemingly-natural inequality was one of the motivations of
the theory of  \cite{BT1989,BT1992}, and does not hold
for some of the previously-defined candidates of large-scale dimension in the
literature \cite{Naudts,Naudts2}. 

Let $X$ denote the simple symmetric random walk on $\Z^d$
where $d\ge 3$.
According to the local central limit theorem,
$g(x\,,y)\sim\text{const}\cdot\|x-y\|^{2-d}$ as $\|x-y\|\to\infty$.
Therefore, Proposition \ref{pr:BT} applies and implies the very appealing
fact that the macroscopic Hausdorff dimension of the range of $X$ 
is a.s.\ $2$. 

Barlow and Taylor \cite{BT1992}
have proved that macroscopic Hausdorff dimension of
the range of transient Brownian motion is also a.s.\ 2, giving further credance to their assertion
that $\Dim$ is a natural large-scale variation of the classical
notion of [microscopic] Hausdorff dimension, nowadays usually denoted
by $\dim_{\rm H}$.

\subsection{Recurrent sets for Markov chains}\label{subsec:RecMC}
We follow the existing works on  related potential theory 
\cite{BPP1995,Bucy1965,FS1989,Lam1963},
and consider a somewhat more general setting in which our random
walk $X$ is replaced by a transient Markov chain, still denoted by $X$. However,
by contrast with the prior works
we continue to assume that our Markov chain $X$ takes values in 
the special state space $\Z^d$
for some $d\ge 1$, and not a general countable state space.

We will continue to write
$\mathcal{R}_X$ to stand for the range of the
Markov chain $X$. That is, $\mathcal{R}_X$ is the random set
$\cup_{n=0}^\infty\{X_n\}$.

Let $P^a$ continue to denote the conditional law of $X$, given $X_0=a$.
Then we say that a random or nonrandom set $F\subseteq\Z^d$ is \emph{recurrent}
for $X$ under $P^a$ when $\mathcal{R}_X\cap F$ is unbounded $P^a$-a.s.

We are aware of at least two characterizations of recurrent nonrandom sets
for general chains, due to Bucy \cite{Bucy1965} and Benjamini et al \cite{BPP1995}.
In order to describe the second characterization, which turns out to be more relevant to
our needs, let $M_1(F)$ denote the collection of all
probability measures on $F$, and $c_1(F\,;a)$  the 
\emph{Martin capacity} of $F$ for the walk started
at $a\in\Z^d$ \cite{BPP1995}. That is,
\begin{equation}\label{c1}
	c_1(F\,;a) := \sup_{\substack{F_0\subseteq F:\\
	F_0\text{ \rm  finite}}}\left[
	\inf_{\mu\in M_1(F)}\mathop{\sum\sum}_{\substack{x,y\in\Z^d:\\
	g(a,y)>0}}
	\frac{g(x\,,y)}{g(a\,,y)} \mu(x)\,\mu(y)\right]^{-1},
\end{equation}
where $\mu(w):=\mu(\{w\})$ for all $w\in\Z^d$.

Benjamini et al \cite{BPP1995} have characterized all recurrent
sets for transient Markov chains on countable state spaces. If we
apply their result to transient  Markov chains on $\Z^d$,
then we obtain the following:

\begin{proposition}[Benjamini, Pemantle, and Peres \cite{BPP1995}] 
	Choose and fix some $a\in\Z^d$.
	A nonrandom set $F\subseteq\Z^d$ is recurrent for
	$X$ under $P^a$ if and only if $\inf c_1(G\,;a)>0$, where
	the infimum is taken over all cofinite subsets $G$ of $F$.
\end{proposition}

Recall that a set $G\subset\Z^d$
is said to be \emph{cofinite} when $\Z^d\setminus G$ is a bounded set.

The preceding capacity condition for cofinite sets is not so easy to verify in
concrete settings. There is an older result, due to
Lamperti \cite{Lam1963}, which contains a more easily-applicable 
characterization of recurrent
sets for ``nice'' Markov chains. The following is a slightly different
formulation that works specifically for transient chains on $\Z^d$.
Barlow and Taylor \cite[Proposition 8.2]{BT1992} state a special case
of it by adapting Lamperti's method [see Example \ref{ex:RW:nice} below].
We derive it later on as a corollary of a ``master theorem'' on hitting
probabilities of transient chains on $\Z^d$ [Theorem \ref{th:RW:FP}].

\begin{corollary}[Lamperti's test]\label{cor:Lamperti}
	Suppose that there exist $a\in\Z^d$ and a finite constant $K>0$ 
	such that for all $n\geq K$ and $m\geq n+K$, 
	\begin{equation}\label{cond:Lamperti}
		\sup_{\substack{x\in\cS_n\\y\in\cS_m}} \frac{g(x\,,y)}{g(a\,,y)}+
		\sup_{\substack{x\in\cS_m\\y\in\cS_n}} \frac{g(x\,,y)}{g(a\,,y)}\leq K.
	\end{equation}
	Then, $F$ is recurrent for $X$ under $P^a$  if and only if
	$\sum_{k=0}^\infty c_1(F\cap\cS_k;a)=\infty$.
\end{corollary}

\begin{example}\label{ex:RW:nice}
	Suppose that $X$ is a random walk that satisfies
	Condition \eqref{cond:BT} of Proposition \ref{pr:BT}. It readily
	follows that $g(0\,,y)>0$ for all $y\neq 0$,
	and\footnote{As usual, we write
	``$f(z)\asymp g(z)$ for all $z\in Z$'' to mean that 
	there exists a positive and finite constant $C$
	such that $C^{-1} g(z) \le f(z)\le Cg(z)$ for all $z\in Z$.}
	\begin{equation}
		\mathop{\sum\sum}_{x,y\in\cS_k}
		\frac{g(x\,,y)}{g(0\,,y)} \mu(x)\,\mu(y)
		\asymp 2^{k(d-\alpha)}
		\mathop{\sum\sum}_{x,y\in\cS_k} g(x\,,y)\,\mu(x)\,\mu(y),
	\end{equation}
	simultaneously for all integers $k\ge 0$ and $\mu\in M_1(F)$. Therefore, 
	$c_1(F\cap \cS_k\,;0) \asymp 2^{-k(d-\alpha)}
	\text{cap}_g(F\cap\cS_k)$ for all $k\ge 0$,
	where 
	\begin{equation}
		\text{cap}_g(G) := \left[\inf_{\mu\in M_1(G)}
		\mathop{\sum\sum}_{x,y\in G} g(x,y)\,\mu(x)\,\mu(y)
		\right]^{-1}
	\end{equation}
	describes the usual random-walk capacity of $G\subset\Z^d$.\footnote{%
	Standard last-exit arguments, and/or maximum principle arguments,
	show that our ``$\text{cap}_g$''
	is the same capacity form as Lamperti's ``$C$''
	\cite{Lam1963} and Barlow and Taylor's ``$\text{Cap}_G$''
	\cite{BT1992}. This fact can be found implicitly in
	Bucy \cite{Bucy1965}, and might be older still.}
	It is easy to see that the Lamperti-type condition \eqref{cond:Lamperti}
	also holds in this case since the walk has i.i.d.\ increments. 
	Therefore, Corollary \ref{cor:Lamperti} tells us
	that $F$ is recurrent for $X$ under $P^0$ iff 
	$\sum_{k=0}^\infty 2^{-k(d-\alpha)} \text{cap}_g(F\cap\cS_k)=\infty$.
	This is Proposition 8.2 of \cite{BT1992}.
\end{example}

\section{Macroscopic fractal percolation}
\label{sec:fp}

We temporarily leave the topic of Markov chains and random walks  
in order to present some basic facts about macroscopic fractal percolation.

Let $k\ge 0$ denote a fixed integer,
and suppose $\{U(Q)\}_{Q\in\D\cap\cV_k}$ is a collection of independent
random variables, defined on a rich enough probability space
$(\Omega\,,\mathcal{F},\P)$, such that each $U(Q)$
is distributed uniformly between $0$ and $1$.
We may define, for all $p\in(0\,,1]$,
\begin{equation}
	\label{eq:ind}
	I_p(Q) := \1_{(0,p)} (U(Q)).
\end{equation}
Then: 
\begin{enumerate}
\item[(i)] $\{I_p(Q)\}_{Q\in\D\cap\cV_k}$ are i.i.d.;
\item[(ii)] $\P\{I_p(Q)=1\}=p$ and $\P\{I_p(Q)=0\}=1-p$; and
\item[(iii)] $I_{p_1}(Q)\le I_{p_2}(Q)$ if $p_1\le p_2$.
\end{enumerate}

For all integers $n\ge 0$ define 
\begin{equation}
	\Pi_{p,n}(\cV_k)  := \bigg\{ Q\in\D_{k-n}: \, Q\subseteq\cV_k,\
	I_p(Q')=1
	\text{ if }Q\subseteq Q'\in\bigcup_{j=k-n}^{k+1}\D_j\bigg\}.
\end{equation}
We also abuse notation to define 
$\Pi_{p,-1}(\cV_k) := \cV_k$.  In this way, we see for example that $\Pi_{p,0}$ 
is a random set of cubes with sidelength $2^k$ which  
result from the \textit{first step} of a certain branching process;
see Figure \ref{fig:treerep}.

\emph{Fractal percolation} on $\cV_k$ [with parameter $p\in(0\,,1]$]
is the random set
\begin{equation}
\label{eq:piinf}
	\Pi_{p,\infty}(\cV_k) := \bigcap_{n=0}^\infty\Pi_{p,n}(\cV_k).
\end{equation}
One can see that this is the usual construction of Mandelbrot's
fractal percolation \cite{Mandelbrot},
scaled to take place in the cube $\cV_k$. Namely, we may write $\cV_k$ as a disjoint
union of $2^d$ elements of $\D_k$; each of those elements
is selected independently 
with probability $p$ and rejected with probability $1-p$. We then write
every one of the selected cubes as a disjoint union of $2^d$ elements
of  $\D_{k-1}$; each resulting
subcube is kept/selected or discarded/deselected independently with respective probabilities
$p$ and $1-p$; and we continue. 

Elementary branching-process theory implies
that $\P\{\Pi_{p,\infty}(\cV_k)\neq\varnothing\}>0$ if and only if $p>2^{-d}$;
see also Figures \ref{fig:cubeperc0}, \ref{fig:sim0.8}, and \ref{fig:treerep}.

Presently, we are interested in performing fractal percolation in $\cV_k$ but we
will stop the subdivisions after $k+1$ steps. In other words, we are interested
in $\Pi_{p,k}(\cV_k)$, which is a random, possibly empty, collection of
side-one cubes in $\D_0$. Since $I_{p_1}(Q)\le
I_{p_2}(Q)$ whenever $p_1\le p_2$, we see that 
$\Pi_{p_1,k}(\cV_k)\subseteq\Pi_{p_2,k}(\cV_k)$ a.s., and hence
if $p_1\le p_2$, then
\begin{equation}
	\P\{ \Pi_{p_1,k}(\cV_k)\cap F\neq\varnothing\}
	\le \P\{ \Pi_{p_2,k}(\cV_k)\cap F\neq\varnothing\}
\end{equation}
for all Borel sets $F\subset\R^d$.

Now let us construct all of these fractal percolations on the same probability space
so that: 
	\begin{enumerate}
		\item[(i)]  $\Pi_{p,k}(\cV_k)$ is a fractal percolation in $\cV_k$ for
			every $k\ge 0$ as decsribed earlier,
		\item[(ii)] $\Pi_{p,0}(\cV_0),\,\Pi_{p,1}(\cV_1)\,,\cdots$
			are independent. 
	\end{enumerate}
In words, we appeal to the preceding procedure in order to construct the
$\Pi_{p,k}(\cV_k)$'s simultaneously for all $k$, using an independent collection of weights
$I_p(Q)$'s for every $\cV_k$.

By \emph{macroscopic fractal percolation} we mean the random set
\begin{equation}
	\Pi_p := \bigcup_{k=0}^\infty (\Pi_{p,k}(\cV_k)\cap \cS_k).
\end{equation}

\begin{figure}[h!]
	\begin{center}
		\begin{tikzpicture}
			\fill[ color=nicos-red](0,0)rectangle(4,4);
			\draw[fill, color=white](0,2)rectangle(2,4);
			\draw(0,0)rectangle(4,4);
			\draw(0,2)--(4,2);
			\draw(2,0)--(2,4);
		\end{tikzpicture}
		\hspace{0.5 in}
		\begin{tikzpicture}
			\fill[ color=nicos-red](0,0)rectangle(4,4);
			\draw[fill, color=white](0,2)rectangle(2,4);
			
			\draw[fill, color=white](0,1)rectangle(2,2);
			\draw[fill, color=white](2,0)rectangle(3,1);
			\draw[fill, color=white](2,2)rectangle(3,3);
			\draw[fill, color=white](3,3)rectangle(4,4);
			
			\draw(0,0)rectangle(4,4);
			\draw(0,2)--(4,2);
			\draw(2,0)--(2,4);
			
			\draw(0,1)--(4,1);
			\draw(3,0)--(3,4);
			\draw(2,3)--(4,3);
			\draw(1,0)--(1,2);
		\end{tikzpicture}
		\vspace{0.25 in}
		
		\begin{tikzpicture}
			\fill[ color=nicos-red](0,0)rectangle(4,4);
			\draw[fill, color=white](0,2)rectangle(2,4);
			
			\draw[fill, color=white](0,1)rectangle(2,2);
			\draw[fill, color=white](2,0)rectangle(3,1);
			\draw[fill, color=white](2,2)rectangle(3,3);
			\draw[fill, color=white](3,3)rectangle(4,4);
			
			\draw[fill, color=white](0,0.5)--(0,1)--(1,1)--(1,0)--(0.5,0)--(0.5,0.5)--(0,0.5);
			\draw[fill, color=white](1,0)rectangle(2,.5);
			\draw[fill, color=white](3,0)--(3,.5)--(3.5,.5)--(3.5,1)--(4,1)--(4,0)--(3,0);
			\draw[fill, color=white](2,1)--(2,2)--(2.5,2)--(2.5,1.5)--(3.5,1.5)--(3.5,1)--(2,1);
			\draw[fill, color=white](3,2)rectangle(3.5,3);
			\draw[fill, color=white](2.5,3.5)rectangle(3,4);

			\draw(0,0)rectangle(4,4);
			\draw(0,2)--(4,2);
			\draw(2,0)--(2,4);
			
			\draw(0,1)--(4,1);
			\draw(3,0)--(3,4);
			\draw(2,3)--(4,3);
			\draw(1,0)--(1,2);
			
			\draw(0,0.5)--(2,.5) (3,.5)--(4,.5);
			\draw(2,1.5)--(4,1.5);
			\draw(3,2.5)--(4,2.5);
			\draw(2,3.5)--(3,3.5);
			
			\draw(0.5,0)--(.5,1);
			\draw(1.5,0)--(1.5,1);
			\draw(2.5,1)--(2.5,2) (2.5, 3)--(2.5,4);
			\draw(3.5,0)--(3.5,3);

		\end{tikzpicture}
		\hspace{0.5 in}
		\begin{tikzpicture}
			\fill[ color=nicos-red](0,0)rectangle(4,4);
			\draw[fill, color=white](0,2)rectangle(2,4);
			
			\draw[fill, color=white](0,1)rectangle(2,2);
			\draw[fill, color=white](2,0)rectangle(3,1);
			\draw[fill, color=white](2,2)rectangle(3,3);
			\draw[fill, color=white](3,3)rectangle(4,4);
			
			\draw[fill, color=white](0,0.5)--(0,1)--(1,1)--(1,0)--(0.5,0)--(0.5,0.5)--(0,0.5);
			\draw[fill, color=white](1,0)rectangle(2,.5);
			\draw[fill, color=white](3,0)--(3,.5)--(3.5,.5)--(3.5,1)--(4,1)--(4,0)--(3,0);
			\draw[fill, color=white](2,1)--(2,2)--(2.5,2)--(2.5,1.5)--(3.5,1.5)--(3.5,1)--(2,1);
			\draw[fill, color=white](3,2)rectangle(3.5,3);
			\draw[fill, color=white](2.5,3.5)rectangle(3,4);
			
			\draw[fill=white](0.25,0.25)rectangle(0.5,0.5);
			\draw[fill=white](1.25,0.75)rectangle(1.75,1);
			\draw[fill=white](3.25,0.75)rectangle(3.5,0.5);
			\draw[fill=white](2.75,2)rectangle(3,1.5);
			\draw[fill=white](3.5,1.5)rectangle(3.75,1.75);
			\draw[fill=white](3.5,2.5)rectangle(4,3);
			\draw[fill=white](2.5,3.75)--(2.25,3.75)--(2.25,3.25)--(2.75, 3.25)--(2.75,3.5)--(2.5,3.5)--(2.5,3.75);

			\draw(0,0)rectangle(4,4);
			\draw(0,2)--(4,2);
			\draw(2,0)--(2,4);
			
			\draw(0,1)--(4,1);
			\draw(3,0)--(3,4);
			\draw(2,3)--(4,3);
			\draw(1,0)--(1,2);
			
			\draw(0,0.5)--(2,.5) (3,.5)--(4,.5);
			\draw(2,1.5)--(4,1.5);
			\draw(3,2.5)--(4,2.5);
			\draw(2,3.5)--(3,3.5);
			
			\draw(0.5,0)--(.5,1);
			\draw(1.5,0)--(1.5,1);
			\draw(2.5,1)--(2.5,2) (2.5, 3)--(2.5,4);
			\draw(3.5,0)--(3.5,3);
			
			\draw(3.75,2.5)--(3.75,3);
			\draw(3.5,2.75)--(4,2.75);
			\draw(2.75,1.75)--(3,1.75);
		\end{tikzpicture}
	\end{center}
	\caption{An image of a simulation of the last 4 of the 5 stages of the construction
	of fractal percolation in $\cV_3$, where $p=\nicefrac12$ in this case. The first stage is
	omitted: That stage shows all of $\cV_3$ colored in. 
	The common width of the white cubes encodes the stage at which the cube was deleted.}
	\label{fig:cubeperc0}
\end{figure}
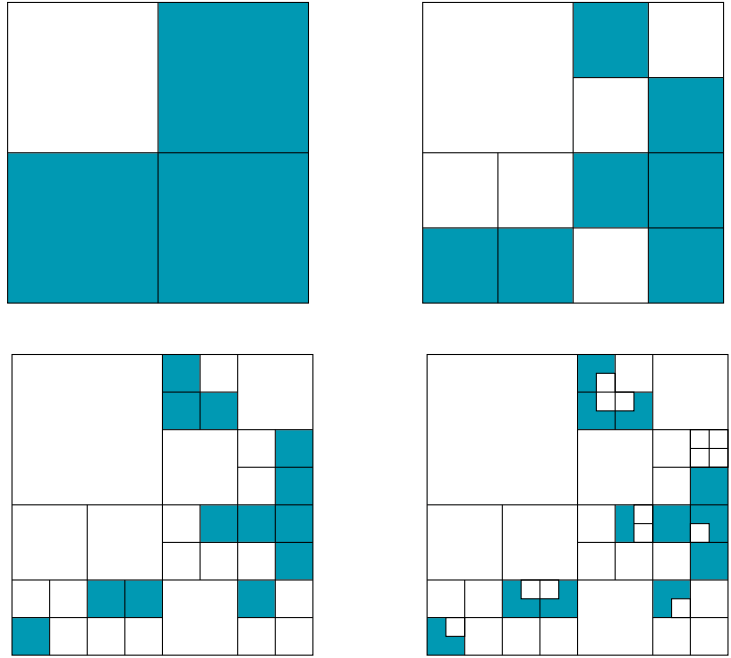

\begin{figure}[h!]
\begin{center}
{\includegraphics[width=6.5cm,height=6.5cm]{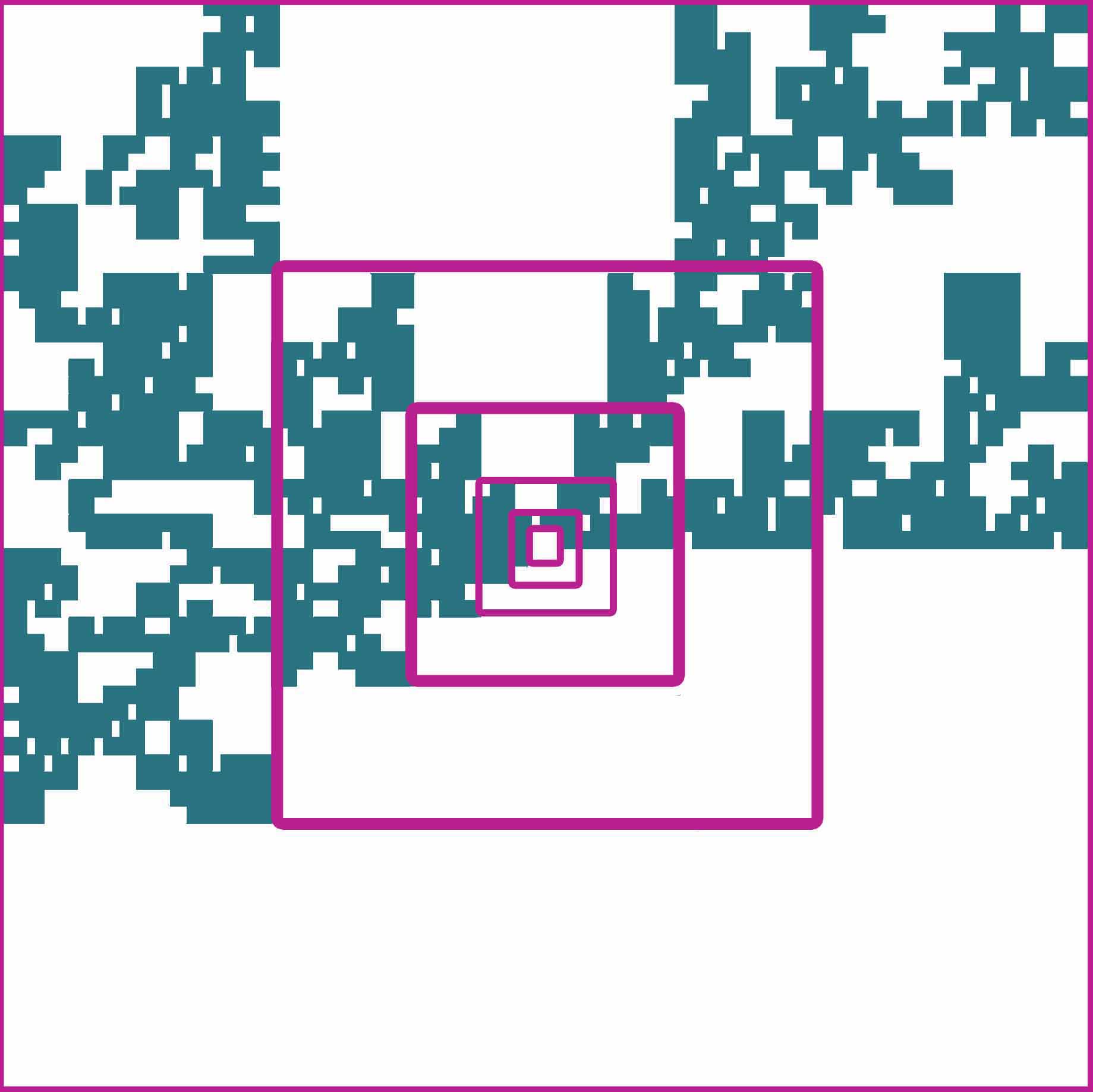}}
\end{center}
\caption{ A simulation of the percolation cluster (dark region) for $\Pi_{p}\cap\cV_5$ for $p=0.8$.
Shells are independent among shells. The nested squares delineate the cubical shells
sets $\cS_0,~\cS_1,~\cS_2,~\cS_3,~\cS_4$ and $\cS_5$. With the exception of $\cS_0$ 
all other shells $\cS_k$ are cubical annuli. Percolation in different cubical annuli results from independent branching processes.}
\label{fig:sim0.8}
\end{figure}

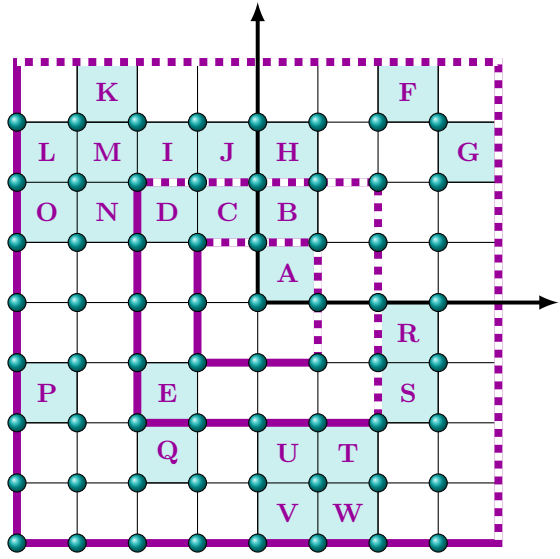
\begin{figure}[h!]
		\begin{center}
			\begin{tikzpicture}[>=latex,scale=0.8]
				\definecolor{sussexg}{rgb}{0,0.7,0.7}
				\definecolor{sussexp}{rgb}{0.6,0,0.6} 
				
				
				\draw[fill = sussexg!20] (0,0)rectangle(1,1);
				\draw[fill = sussexg!20] (2,3)rectangle(3,4);
				\draw[fill = sussexg!20] (-2,-2)rectangle(-1,-1);
				\draw[fill = sussexg!20] (-2,-3)rectangle(-1,-2);
				\draw[fill = sussexg!20] (0,-4)rectangle(2,-2);
				\draw[fill = sussexg!20] (2,-2)rectangle(3,0);
				\draw[fill = sussexg!20] (3,2)rectangle(4,3);
				\draw[fill = sussexg!20] (-4,1)rectangle(1,3);
				\draw[fill = sussexg!20] (-3,3)rectangle(-2,4);
				\draw[fill = sussexg!20] (-4,-2)rectangle(-3,-1);
				\draw[fill = sussexg!20] (0,0)rectangle(1,1);
				\draw[fill = sussexg!20] (0,0)rectangle(1,1);
				\draw[fill = sussexg!20] (0,0)rectangle(1,1);

			\foreach \x in {-4,...,4}{\draw(\x, -4)--(\x,4);}
				\foreach \x in {-4,...,4}{\draw( -4,\x)--(4,\x);}

			
%
				
				\draw[color=sussexp, line width=3pt](-4,-4)rectangle(4,4);
			
				\draw[color=sussexp, line width=3pt](-2,-2)rectangle(2,2);
				
				\draw[color=sussexp, line width=3pt](-1,-1)rectangle(1,1); 
				
				\draw[color=white, dashed, line width=3pt](-1,1)--(1,1)--(1,-1); 
				\draw[color=white, dashed, line width=3pt](-2,2)--(2,2)--(2,-2); 
				\draw[color=white, dashed, line width=3pt](-4,4)--(4,4)--(4,-4); 
				
				\draw[<->, line width=1.5pt](5,0)--(0,0)--(0,5);
				
				\foreach \x in {-4,...,3}{
					\foreach \y in {-4,...,3}{ 
					\draw[shading=ball, ball color=sussexg](\x,\y) circle (1.5mm);
					}}
					
				
				\draw(0.5,0.5)node[color=sussexp]{\textbf{A}};
				\draw(0.5,1.5)node[color=sussexp]{\textbf{B}};
				\draw(-0.5,1.5)node[color=sussexp]{\textbf{C}};
				\draw(-1.5,1.5)node[color=sussexp]{\textbf{D}};
				\draw(-1.5,-1.5)node[color=sussexp]{\textbf{E}};
				\draw(3.5,2.5)node[color=sussexp]{\textbf{G}};
				\draw(2.5,3.5)node[color=sussexp]{\textbf{F}};
				\draw(0.5,2.5)node[color=sussexp]{\textbf{H}};
				\draw(-1.5,2.5)node[color=sussexp]{\textbf{I}};
				\draw(-0.5,2.5)node[color=sussexp]{\textbf{J}};
				\draw(-2.5,3.5)node[color=sussexp]{\textbf{K}};
				\draw(-3.5,2.5)node[color=sussexp]{\textbf{L}};
				\draw(-2.5,2.5)node[color=sussexp]{\textbf{M}};
				\draw(-2.5,1.5)node[color=sussexp]{\textbf{N}};
				\draw(-3.5,1.5)node[color=sussexp]{\textbf{O}};
				\draw(-3.5,-1.5)node[color=sussexp]{\textbf{P}};
				\draw(-1.5,-2.5)node[color=sussexp]{\textbf{Q}};
				\draw(1.5,-2.5)node[color=sussexp]{\textbf{T}};
				\draw(0.5,-2.5)node[color=sussexp]{\textbf{U}};
				\draw(0.5,-3.5)node[color=sussexp]{\textbf{V}};
				\draw(1.5,-3.5)node[color=sussexp]{\textbf{W}};
				\draw(2.5,-0.5)node[color=sussexp]{\textbf{R}};
				\draw(2.5,-1.5)node[color=sussexp]{\textbf{S}};
					
			\end{tikzpicture}
		\end{center}	
	\caption{\emph{Fractal percolation schematic.} Each square 
	is indexed by its lower left, south-west, corner. 
	Percolation is independent between the thickset cubical annuli, 
	and is the result of a coupled branching process in each cube $\cV_k$; i.e.,
	the random sets $\{ A \}$, $\{ B,\ldots, E\}$ and $\{ F, G, \ldots, W \}$ 
	are the result of independent branching processes. 
	\emph{The enumeration scheme.} The annuli are enumerated by their index;
	this indexing forces a partial ordering for the cubes in each shell. 
	Thus, cubes in $\cS_0$ are enumerated before those in $\cS_1$ and so on. 
	Percolation in shell $\cS_k$ is the result of percolation in $\cV_k$ 
	according to the description that follows \eqref{eq:piinf}. 
	Each $\cV_k$ is subdivided into four cubes which we enumerate counterclockwise, 
	starting from the one that corresponds to the the first quadrant. 
	Each of those is divided into four cubes; those are again enumerated counterclockwise. 
	This enumeration procedure is continued until we are left with cubes of size 1. 
	which are now (for the purposes of the figures) enumerated lexicographically. 
	This enumeration scheme of size one cubes is the isomorphism between 
	the percolation set and a forest, and its further illustrated in Figure \ref{fig:forest}.
	}
	\label{fig:cubeperc}
\end{figure}

Of course, $\Pi_0=\varnothing $ and $\Pi_1=\R^d$. 
Starting from here, we often assume tacitly that
$p\in(0\,,1)$ in order to avoid the trivial cases $p=0$ and $p=1$. 
In any case, it is easy to deduce our next result.

\begin{lemma}\label{lm:basic}
	The following are valid:
	\begin{enumerate}
	\item $\Pi_{p_1}\subseteq\Pi_{p_2}$ whenever $p_1\le p_2$;
	\item $\Pi_p\cap\cS_0,
		\Pi_p\cap\cS_1,\Pi_p\cap\cS_2,\ldots$ are independent random sets;
	\item $\Pi_p\cap\cS_k$ is distributed as 
		$\Pi_{p,k}(\cV_k)\cap\cS_k$ for every integer $k\ge 0$.
	\end{enumerate}
\end{lemma}

If $A$ and $B$ are subsets of $\R^d$, then we say that $A$ is
a \emph{recurrent set} for $B$ provided that there exist infinitely many shells
$\cS_{k_1},\cS_{k_2},\ldots$ such that $A\cap B\cap\cS_{k_n}\neq\varnothing$
for all $n\ge 1$. We sometimes write ``$A\Recurrent  B$\,'' in place of ``$A$ is
a recurrent set for $B$\,''; clearly, $\Recurrent $ is an equivalence relation
between pairs of subsets of $\R^d$.

\begin{lemma}\label{lem:0-1}
	If $F\subset\R^d$ is non random, then $\P\{F\Recurrent \Pi_p\}=0$
	or $1$.
\end{lemma}

\begin{proof}
	Let $\zeta_k=1$ if $F\cap \Pi_p\cap \cS_k\neq\varnothing$
	and $\zeta_k=0$ otherwise. Then, the $\zeta_k$'s are independent and
	\begin{equation}
		\P\{F\Recurrent \Pi_p\}=\P\left\{
		\sum_{k=1}^\infty\zeta_k=\infty\right\}\in\{0\,,1\},
	\end{equation}
	by Kolmogorov's $0$-$1$ law.
\end{proof}

A word of notational caution at this point. We use $F\Recurrent \Pi_p$ as 
shorthand for $\P\{F\Recurrent \Pi_p\} = 1$, since $\Pi_p$ is a random set.
However, when we condition on a given configuration from this a.s.\ event, 
 $F\Recurrent \Pi_p$ reverts to its original meaning.


As we have noted already, the [microscopic] fractal percolation set 
$\Pi_{p,\infty}(\cV_k)$ is nonvoid if and only if $p>2^{-d}$.
The large-scale analogue of becoming nonvoid is to become 
unbounded. The following shows that the large-scale result 
takes a different 
form than its small-scale counterpart in the critical case $p=2^{-d}$.

\begin{lemma}\label{lem:perc:unbdd}
	$\Pi_p$ is almost surely unbounded if $p\ge 2^{-d}$ and
	it is almost surely bounded if $p<2^{-d}$.
	\label{lem:20}
\end{lemma}

We will be interested only in $\Pi_p$ when it is unbounded; the
preceding tells us that we want to consider only values of
$p\in[2^{-d}\,,1]$. The said condition on $p$ will appear several
times in the sequel for this very reason.

\begin{proof}
	By the Borel--Cantelli lemma for independent
	events, $\Pi_p$ is unbounded a.s.\ if $\sum_{k=0}^\infty
	\P\{\Pi_p\cap\cS_k\neq\varnothing\}=\infty$; otherwise,
	$\Pi_p$ is bounded a.s. 
	
	Since $\Pi_p\cap\cS_k\subset\Pi_p\cap\cV_k$,
	the probability that $\Pi_p\cap\cS_k$ is nonempty 
	is at most the probability that, in its first $k$ generations,
	a Galton--Watson branching
	process with mean branch rate $2^dp$ survives. We shall denote
	by $Z_k$ the number of descendants in the $k$-th generation.
	When $p<2^{-d}$, the said Galton--Watson process
	is strictly subcritical. It is well-known that 
	$\E(Z_k)=(\E Z_1)^k=(2^dp)^k,$ which forms a summable sequence
	in $k$. The simple bound
	$\P\{Z_k>0\}\le \E(Z_k)$ ensures that
	\begin{equation}
		\sum_{k=1}^\infty\P\{\Pi_{p}\cap\cS_k\neq \varnothing\}
		\le\sum_{k=1}^\infty \P\{\Pi_{p}\cap\cV_k\neq \varnothing\}\le
		\sum_{k=1}^\infty\E(Z_{k+1})<\infty.
	\end{equation}
	Thus, we  conclude that $\Pi_p$ is a.s.\ bounded when $p<2^{-d}$. 
	
	By the monotonicity
	of $p\mapsto\Pi_p$, it remains to prove that $\Pi_{2^{-d}}$ is a.s.\
	unbounded.  One can define fractal percolation
	on any dyadic cube $Q\in\D$ in much the same way as
	we defined it on $\cV_k$. Next we note that if $k\ge 2$, then 
	\begin{equation}
		\P\left\{\Pi_{2^{-d}}\cap\cS_k\neq\varnothing
		\right\}\ge p^2\P(E_k),
	\end{equation}
	where $E_k$ denotes the event that fractal percolation with
	 $p=2^{-d}$ on a dyadic
	cube of side $2^{k-1}$ does not become void in $k$ steps.
	Therefore, $\P\{\Pi_{2^{-d}}\cap\cS_k\neq\varnothing\}$
	is bounded below by the probability that a certain critical Galton--Watson
	does not become extinct in its first $k$ generations. A well-known
	theorem of Kolmogorov \cite{K1938}
	asserts that for critical Galton--Watson branching process, 
	\begin{equation}
		\lim_{k\to\infty}k\P\{Z_k>0\}=\frac{2}{\sigma^2},
	\end{equation}
	where $\sigma^2$ denotes the variance of the offspring distribution;
	see also Kesten et al \cite{KNS1966} and Lyons et al \cite{LPP1995}.  
	This  implies that there exists $k_0>1$ such that
	\begin{equation}
		\sum_{k=0}^{\infty}
		\P\{\Pi_{2^{-d}}\cap{\cal S}_k\neq\varnothing\}
		\ge p^2\sum_{k= k_0}^\infty\P\{Z_k>0\}\ge
		\frac{p^2}{k_0}\sum_{k=k_0}^{\infty} \frac1k=\infty,
	\end{equation}
	and completes the proof.
\end{proof}

The following is the result of an elementary computation.
\begin{lemma}\label{lem:perc:hit}
	If $x\in\cS_k$ for some $k\ge 1$,
	then $\P\{x\in \Pi_p\} = p^{k+1}.$ If $x,y\in\cS_k$
	then $\P\{x,y\in\Pi_p\}= p^{2k+2-{\rm d}(x,y)}$, where
	\begin{equation}\label{d}
		{\rm d}(x\,,y) := (k+1)-\min\left\{ n\ge0:\ \exists Q\in\D_{n}
		\text{ such that }x,y\in Q
		\right\}.
	\end{equation}
\end{lemma}
According to  \eqref{d}, $0 \le \d(x\,,y) \le k+1$ 
when $x,y \in \cS_k$. Both bounds can be achieved:
$\d(x\,,y)=0$ when the most recent common ancestor of $x$ and $y$ 
in the branching process is the root $\cV_k$; and $\d(x\,,y)=k+1$ when $y=x$.

We will use the preceding in order to prove the following.

\begin{theorem}\label{th:Dim:Pi}
	For all non random sets $F\subset\R^d$, 
	\begin{equation}
		\Dim (F) = -\log_2 \inf\left\{ p\in[2^{-d},1]:\, F\Recurrent \Pi_p\right\}.
	\end{equation}
\end{theorem}

\begin{remark}\label{rem:Dim:Pi}
	Theorem \ref{th:Dim:Pi} can be recast in the following
	equivalent terms: If $\Dim(F)>-\log_2p,$ then
	$F\Recurrent\Pi_p$; otherwise if $\Dim(F)<-\log_2p$, then
	$F\nRecurrent\Pi_p$. The case that $\Dim(F)=-\log_2p$
	is not decided by a dimension criterion. That case can be decided by
	a more delicate capacity criterion.
	However, we will not have need for this  refinement,
	and so will not pursue it here.
\end{remark}

\begin{proof}
	Define, as in Barlow and Taylor \cite[p.\ 137]{BT1992},
	$F^z=\phi(F)$, where $\phi(x)$ denotes the closest point
	in $\Z^d$ to $x\in\R^d$, with some procedure in place for
	breaking ties. We may observe that $F\cap\Pi_p\cap\cS_k$
	is nonempty if and only if $F^z\cap\Pi_p\cap\cS_k$ is nonempty.
	Lemma 6.1 of Barlow and Taylor \cite{BT1992} ensures that
	$\Dim(F)=\Dim(F^z)$. Therefore, we may
	assume without loss of generality that $F$ is a subset of $\Z^d$;
	otherwise, we can replace $F$ by $F^z$ everywhere 
	throughout the remainder of the
	proof. From now on we consider only $F\subseteq\Z^d$.
	
	Let $k\ge 0$ be an arbitrary integer.	
	We can find dyadic cubes $Q_1,\ldots,Q_m\subset\cS_k$ such that:
	\begin{enumerate}
		\item The sidelength $2^{\ell_i}$ $(0 \le \ell_i \le k+1)$ of every $Q_i$ is at least one;
		\item $(F\cap\cS_k)\subseteq\cup_{i=1}^m Q_i$; and
		\item $\mathcal{N}_{\log_2(1/p)}(F\,,\cS_k) = \sum_{i=1}^m
			p^{k-\ell_i+1}$.
	\end{enumerate}
	Thus, we obtain
	\begin{equation}\begin{split}
		\P\{\Pi_p\cap Q_i\not = \varnothing\}&=
			\P\{\Pi_{p,k}({\cal V}_k)\cap Q_i\not = \varnothing\}\\
		&\le\P\{\Pi_{p,\ell_i-1}({\cal V}_k)\cap Q_i\not = \varnothing\}\\
		&=p^{k +1 -\ell_i}.
	\end{split}\end{equation}
	For the inequality we have used the fact that 
	$\Pi_{p,i}(\cV_k)\supset\Pi_{p,i+1}(\cV_k)$ for every integer $i\ge-1$.
	In any case, it follows that
	\begin{equation}\label{eq:UB:Hauss}
		\P\{ F\cap\Pi_p\cap\cS_k\neq\varnothing\}
		\le \sum_{i=1}^m \P\{\Pi_p\cap Q_i\neq\varnothing\}
		\le\mathcal{N}_{\log_2(1/p)}(F\,,\cS_k).
	\end{equation}
	The preceding holds for all $p\in(0\,,1]$.
 	Suppose for the moment that
	$\log_2(1/p)>\Dim(F)$. Then,
	$\sum_{k=0}^\infty\mathcal{N}_{\log_2(1/p)}
	(F\,,\cS_k)<\infty$ by the definition of $\Dim$; and \eqref{eq:UB:Hauss}
	implies that 
	\begin{equation}
		\sum_{k=0}^\infty\P\{F\cap\Pi_p\cap\cS_k\neq
		\varnothing\}<\infty.
	\end{equation}
	The non recurrence of $F$ for $\Pi_p$ follows by
	the Borel--Cantelli lemma. 
	
	We have proved that
	if $p\in( 0\,, 2^{-\Dim(F)})$, then $F$ is
	not recurrent for $\Pi_p$. 	
	It now remains to show that
	\begin{equation}\label{eq:goal:LB}
		\text{If $\Dim(F)>\delta>0$ and $p\in( 2^{-\delta}\,,1)$, then 
		$F\Recurrent\Pi_p$.}
	\end{equation}
	From now on we choose and fix an arbitrary $\delta\in(0\,,\Dim(F))$.
	Note that 
	\begin{equation}\label{coocoo}
		\sum_{k=0}^\infty\mathcal{N}_\delta(F\,,\cS_k)=\infty,
	\end{equation}
	thanks to the definition of $\Dim$. 
		
	According to Theorem 4.2 of Barlow and Taylor \cite{BT1992},
	we can find a sigma-finite measure $\bar\mu$ on $\R^d$ 
	and a positive and finite constant $c$, that depends only
	on the ambient dimension $d$, such that
	\begin{equation}
		\bar\mu(\cS_k)\ge {\cal N}_{\delta}(\cS_k\,,\cS_k)
		\qquad\text{ and}\qquad
		\bar\mu(Q)\le c 2^{\delta(\ell-k-1)},
	\end{equation}
	for all integers $k\ge 0$
	and all dyadic cubes $Q\subset\cS_k$ with sidelength $2^\ell\ge 1$.
	[Barlow and Taylor \cite{BT1992} define a measure on $\cS_k$, 
	that they refer to as $\mu$. The restriction of our $\bar\mu$ to $\cS_k$
	is their measure $\mu$.] Normalize  
	\begin{equation}
		\mu(\cdot) := \frac{\bar\mu(\cdot)}{\bar\mu(\cS_k)},
	\end{equation}
	in order to conclude the following for all integers $k\ge 0$: 
	\begin{enumerate}
		\item[(i)] $\mu(\cS_k) =1$;
		\item[(ii)] $\mu(Q) \le c 2^{\delta(\ell-k-1)}/\mathcal{N}_\delta
			(F\,,\cS_k)$, uniformly for all dyadic cubes $Q\subset\cS_k$
			with sidelength $2^\ell\ge 1$. [This is valid even when
			$\mathcal{N}_\delta(F\,,\cS_k)=0$, provided that we define
			$1\div 0:=\infty$.]
	\end{enumerate}
	
	Define
	\begin{equation}
		\zeta_k := p^{-k-1}\mu(\cS_k\cap\Pi_p)
		=\sum_{x\in\cS_k} \frac{\1_{\{x\in\Pi_p\}}}{p^{k+1}} \mu(x),
	\end{equation}
	where $\mu(x):=\mu(\{x\})$.
	By Lemma \ref{lem:perc:hit}: (i)
	$\E[\zeta_k] =\mu(\cS_k) =1$;
	and (ii)
	\begin{equation}\begin{split}
		\E[\zeta_k^2] &= \sum_{x,y\in\cS_k} p^{-\d(x,y)}\,\mu(x)\,\mu(y)\\
		&= \sum_{j=0}^{k+1}  p^{-j}\,(\mu\times\mu)\left\{
			(x\,,y)\in\cS_k^2:\ \d(x\,,y)=j\right\}\\
		&\le \sum_{j=0}^{k+1}  p^{-j}\,(\mu\times\mu)\left\{
			(x\,,y)\in\cS_k^2:\ \d(x\,,y)\ge j\right\}.
	\end{split}\end{equation}
	For a given $x$, define $C_j(x) := \{y\in\cS_k:\, \d(x\,,y)\ge j\}$.
	Because $\mu(\cS_k)=1$,
	\begin{equation}
		(\mu\times\mu)\{(x\,,y)\in\cS_k^2:\, \d(x\,,y)\ge j\}
		\le \max_{x\in\cS_k}\mu(C_j(x)).
	\end{equation}
	Since $C_j(x)\subset\D_{k+1-j}$, we can find a
	dyadic cube $Q\in\D_{k+1-j}$ such that $C_j(x)\subset Q$.
	Therefore, property (ii) of the measure 
	$\mu$ ensures that
	\begin{equation}
		\mu(C_j(x))\le \mu(Q) \le \frac{c}{2^{\delta j}\mathcal{N}_\delta(F\,,\cS_k)},
	\end{equation}
	and hence
	\begin{equation}\begin{split}
		\E[\zeta_k^2] &\le \frac{c}{\mathcal{N}_\delta(F\,,\cS_k)}\cdot
			\sum_{j=0}^k  \big(2^{\delta}p\big)^{-j}\le\frac{c\left(1- (2^\delta p)^{-1} \right)^{-1}}{\mathcal{N}_\delta(F\,,\cS_k)}.
	\end{split}\end{equation}
	For the last inequality we have used the hypothesis of (\ref{eq:goal:LB}). 
	The Paley-Zygmund inequality yields
	$\P\{\zeta_k>0\}\ge (\E[\zeta_k])^2/\E[\zeta_k^2]$.
	Therefore, our bounds for
	the two first moments of $\zeta_k$ lead us to
	\begin{equation}
		\P\{\Pi_p\cap F\cap \cS_k\neq\varnothing\} 
		\ge \mathcal{N}_\delta(F\,,\cS_k)\cdot{\frac{\left(1- (2^\delta p)^{-1} \right)}%
		{c}}.
	\end{equation}
	This and \eqref{coocoo} together imply that
	$\sum_{k=0}^\infty\P\{\Pi_p\cap F\cap \cS_k\neq\varnothing\} =\infty$.
	The independence half of the Borel--Cantelli lemma implies \eqref{eq:goal:LB}.
\end{proof}

Let us close this section by presenting a quick application of Theorem 
\ref{th:Dim:Pi}.

\begin{corollary}\label{co:Dim:Pi}
	For each $p\in(0\,,1]$, $\Dim(\Pi_p)=(d+\log_2 p)^+$ a.s.
\end{corollary}

This result  has content only when $p\in[2^{-d}\,,1]$; see Lemma \ref{lem:20}.
In particular, it states that the dimension of fractal percolation is $0$
at criticality.

\begin{proof}
	 The proof uses the replica argument of Peres \cite{Peres1996} without need
	for  essential changes. More specifically,
	let $\Pi_q'$ denote a fractal percolation set with parameter $q\in(0\,,1]$
	such that $\Pi_p$ and $\Pi_q'$ are independent. Because $\Pi_p\cap\Pi_q'$
	has the same distribution as $\Pi_{pq}$,
	it follows that $\P\{ \Pi_p\Recurrent\Pi_q' \}=\P\{\Pi_{pq}
	\text{ is unbounded}\}. $
	Thus, by Lemma \ref{lem:perc:unbdd}, we have
	\begin{equation}
		\P\{\Pi_p\Recurrent\Pi_q'\}=\begin{cases}
			1&\text{if $p\ge q^{-1}2^{-d}$},\\
			0&\text{if $p< q^{-1}2^{-d}$}.
		\end{cases}
	\end{equation}
	We may first condition on $\Pi_p$ and then appeal to Theorem \ref{th:Dim:Pi}
	in order to see that  $\Dim(\Pi_p)= -\log_2(q_c)$ where $q_c$ is
	the critical constant $q\in(0\,,1]$ such that $pq2^d\ge 1$. Because 
	$q_c =p^{-1}2^{-d}$ the corollary follows.
\end{proof}

We can now deduce Barlow and Taylor's dimension theorem
[Proposition \ref{pr:BT}] from the previous results of this paper.
The following is a standard codimension argument; see Taylor 
\cite[Theorem 4]{Tay1966}.

\begin{proof}[Proof of Proposition \ref{pr:BT}]
	Barlow and Taylor \cite[Cor.\ 8.4]{BT1992} 
	have observed that, under
	the conditions of Proposition \ref{pr:BT},
	$P^a\{\mathcal{R}_X\Recurrent F\}=1$ if $\Dim(F)>d-\alpha$
	and $P^a\{\mathcal{R}_X\nRecurrent F\}=1$ if
	$\Dim(F)<d-\alpha$. This is an immediate consequence of 
	Proposition 8.2 of \cite{BT1992}, which is a variation of Lamperti's
	test \cite{Lam1963}; the general form of this form of Lamperti's test
	is in fact Corollary \ref{cor:Lamperti}, which we will prove in due time. 
	
	We apply the preceding observation conditionally, with
	$F:=\Pi_p$, where {$p\in(2^{-d}\,,1]$} is a fixed parameter. 
	By Corollary \ref{co:Dim:Pi}, $\Dim(\Pi_p)=d+\log_2p$ a.s..
	Therefore, $P^a\{\mathcal{R}_X\Recurrent\Pi_p\}=1$ if $p>2^{-\alpha}$
	 and $P^a\{\mathcal{R}_X\nRecurrent\Pi_p\}=1$ if $p<2^{-\alpha}$.
	
	By the Hewitt--Savage 0-1 law, $\Dim(\mathcal{R}_X)$
	is $P^a$-a.s.\ a constant. Choose  $p > 2^{-\alpha}$ and assume 
	to the contrary that $P^a\{\Dim(\mathcal{R}_X) < -\log_2p\} = 1. $
	Restrict the probability space 
	of the random walk to the full-$P^a$ 
	probability event $\{\Dim(\mathcal{R}_X) < -\log_2p\}$ and fix a 
	realization  $F = \mathcal{R}_X$.
	By Theorem \ref{th:Dim:Pi} we conclude that $\P\{ F \nRecurrent \Pi_p\} = 1$ 
	for almost all 
	realizations of the random walk. This gives the desired contradiction since for 
	$\P$-a.e.\ realization of $\Pi_p$, $P^a\{\mathcal{R}_X\Recurrent\Pi_p\}=1$, while
		\begin{align*}
			1 &= \int \d\P \int \d P^a\ \1\{ \mathcal{R}_X\Recurrent\Pi_p \}
		   	= \int \d P^a \int \d\P\, \1\{ \mathcal{R}_X\Recurrent\Pi_p \}= 0.
		\end{align*}	
	It follows that $\Dim(\mathcal{R}_X) \ge -\log_2p$ a.s.\ as long as $p>2^{-\alpha}$.
	An analogous argument shows that
	$\Dim(\mathcal{R}_X)\le -\log_2p$ whenever $p<2^{-\alpha}$.
	Therefore, we conclude that $\Dim(\mathcal{R}_X)=-\log_2(p_c)$
	$P^a$-a.s.\
	where $p_c = 2^{-\alpha}$. 
\end{proof}

\section{ A forest representation of $\Z^d$}\label{sec:forest}
If $x\in\Z^d\cap\cV_k$ for some integer $k\ge 0$,
then there exists a unique sequence $Q_0(x),Q_1(x),\ldots,Q_{k+1}(x)$
of dyadic sets such that:
\begin{enumerate}
	\item $Q_0(x)=\cV_k$ and $Q_{k+1}(x)=[x_1\,,x_1+1)\times\cdots\times
		[x_d\,,x_d+1)$;
	\item $Q_{i+1}(x)\subset Q_i(x)$ for all $i=0,\ldots,k$; and
	\item $Q_i(x)\in\D_{k-i+1}$ for all $i=0,\ldots,k+1$.
\end{enumerate}
Conversely, if  $Q_0,Q_1,\ldots,Q_{k+1}$ denotes a collection of
dyadic cubes such that:
\begin{enumerate}
	\item $Q_0=\cV_k$;
	\item $Q_{i+1}\subset Q_i$ for all $i=0,\ldots,k$; and
	\item $Q_i\in\D_{k-i+1}$ for all $i=0,\ldots,k+1$;
\end{enumerate}
then there exists a unique point $x\in\Z^d\cap\cV_k$
defined unambiguously via $Q_{k+1}= [x_1\,,x_1+1)\times\cdots\times[
x_d\,,x_d+1)$ [equivalently, $x_i:=\inf_{y\in Q_{k+1}}y_i$
for $1\le i\le d$] (see Figure \ref{fig:treerep}). 
Moreover, $Q_i = Q_i(x)$ for all $0\le i\le k+1$.

The preceding describes a bijection between the points
in $\Z^d\cap\cS_k$ and a certain collection of $(k+2)$-chains of dyadic cubes.
We can now use this bijection in order to
build a directed-tree representation of $\Z^d\cap\cS_k$: 
The vertices of the tree are
comprised of all dyadic cubes $Q\in\D$ whose
sidelength is $\ge 1$. For the [directed] edges of our tree, we draw
an arrow from a vertex $Q$ to a vertex $Q'$ if and only if
there exists an integer $i=0,\ldots,k$ and a point $x\in\Z^d$
such that $Q=Q_i(x)$ and $Q'=Q_{i+1}(x)$. The resulting
graph is denoted by $\mathcal{T}_k$. 

It is easy to observe the following
properties of $\mathcal{T}_k$:
\begin{enumerate}
	\item $\mathcal{T}_k$ is a finite rooted tree, the  root of $\mathcal{T}_k$ 
		being $\cV_k$;
	\item Every ray in $\mathcal{T}_k$ has depth $k+1$;
	\item There is a canonical bijection from the
		rays of $\mathcal{T}_k$ to $\Z^d\cap\cS_k$.
\end{enumerate}

\begin{figure}[h]
		\begin{center}
			\begin{tikzpicture}[>=latex, yscale=0.8]
				\definecolor{sussexg}{rgb}{0,0.8,0.7}
				\definecolor{sussexp}{rgb}{0.6,0,0.6}

				\draw[fill=sussexg!30] (2,4) -- (6,3) -- (10,4)--(6, 5)--(2, 4);

					\draw[color = white, line width=2pt] (3,3.75)--(7,4.75);
					\draw[color = white, line width=2pt] (5,3.25)--(9,4.25);
					
					\draw[color = white, line width=2pt] (3,4.25)--(7,3.25);
					\draw[color = white, line width=2pt] (5,4.75)--(9,3.75);	
					
					\draw[color = white, line width=2pt] (4,3.5)--(8,4.5);
					\draw[color = white, line width=2pt] (4,4.5)--(8,3.5);	
					
					\draw[->, line width=0.1pt] (2,3)--(10,5)node[below left]{$y$};
					\draw[->, line width=0.1pt] (2,5)--(10,3)node[below left]{$x$};

					\draw[color=sussexp, densely dashed, line width=1.2pt](4,7)--(3,4); 
					\draw[color=sussexp, densely dashed,line width=1.2pt](4,7)--(3.75, 3.75); 
					\draw[color=sussexp, densely dashed,line width=1.2pt](4,7)--(4.75,4); 
					\draw[color=sussexp, densely dashed,line width=1.2pt](4,7)--(4, 4.25); 
					
					\draw[color=sussexp, densely dashed,line width=1.2pt](5.7,6.5)--(4.75,3.5); 
					\draw[color=sussexp, densely dashed,line width=1.2pt](5.7,6.5)--(5.75, 3.25); 
					\draw[color=sussexp,densely dashed, line width=1.2pt](5.7,6.5)--(6, 3.75); 
					\draw[color=sussexp, densely dashed,line width=1.2pt](5.7,6.5)--(6.75, 3.5);


					\draw[color=sussexp, line width=2pt](8,7)--(7,4); 
					\draw[color=sussexp, densely dashed, line width=1.2pt](8,7)--(7.75,3.75); 
					\draw[color=sussexp, densely dashed,line width=1.2pt](8,7)--(8,4.25); 
					\draw[color=sussexp, densely dashed, line width=1.2pt](8,7)--(9,4);

				\draw[fill=sussexg!30](2,7) -- (6,6) -- (10,7)--(6, 8)--(2, 7);

					
					\draw[color = white, line width=2pt] (4,6.5)--(8,7.5);
					\draw[color = white, line width=2pt] (4,7.5)--(8,6.5);

				
					\draw[color=sussexp, densely dashed,line width=1.2pt](6,10)--(4,7); 
					\draw[color=sussexp,densely dashed, line width=1.2pt](6,10)--(5.7,6.5); 
				        \draw[color=sussexp,densely dashed, line width=1.2pt](6,10)--(6,7.5); 
					\draw[color=sussexp, line width=2pt](6,10)--(8,7); 
				
				\draw[fill=sussexg!30] (2,10) -- (6,9) -- (10,10)node[right]{\large{$\cV_2=Q_0$}}--(6, 11)--(2, 10);

			\end{tikzpicture}
		\end{center}
	\caption{The tree representation of $\cV_2$ with dyadic cubes as nodes. 
	The levels also indicate steps in the percolation branching process.
	The axes are in the picture in order to help give an orientation. 
	Every $1\times1$ square at the lower level is indexed by its south-west corner. 
	The sequence of cubes in the thickset 
	branch of the tree is in descending order $Q_0(0) \supset Q_1(0) \supset Q_2(0)$.}
	\label{fig:treerep}
\end{figure}
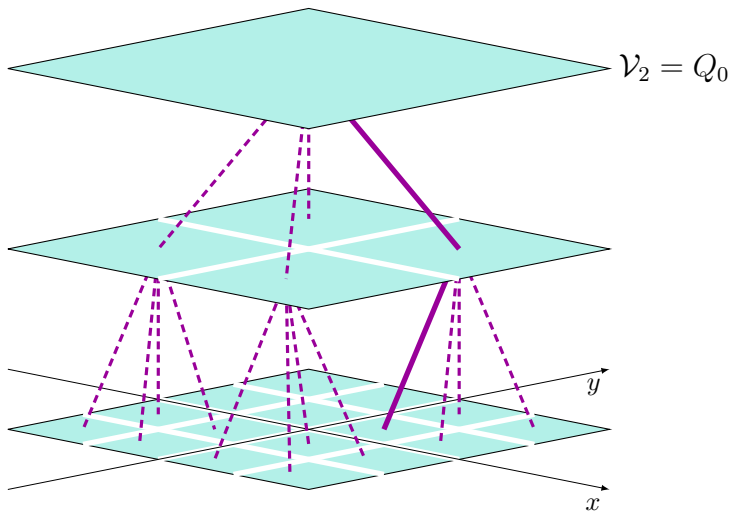

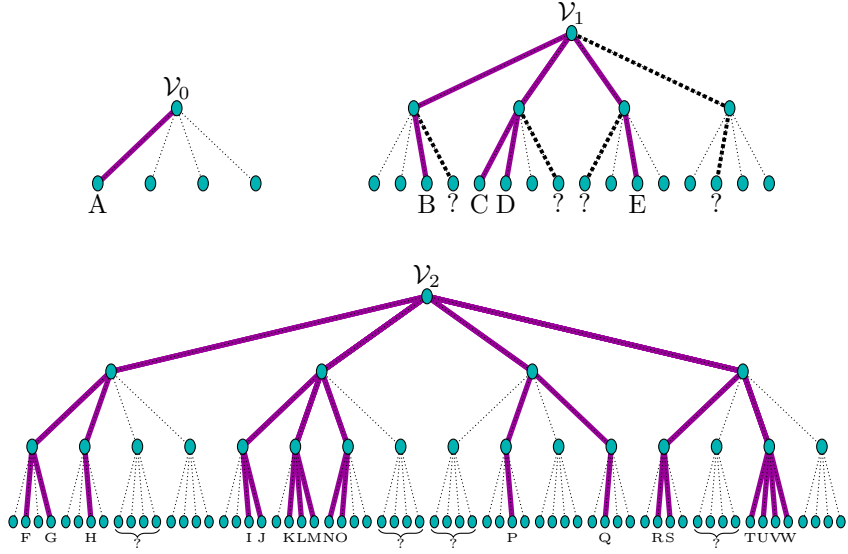
\begin{figure}[h!]
		\begin{center}
			\begin{tikzpicture}[>=latex,xscale=0.7]
				\definecolor{sussexg}{rgb}{0,0.7,0.7}
				\definecolor{sussexp}{rgb}{0.6,0,0.6} 
				
				\draw[color=sussexp, line width=2pt](-0.5,1)--(-2,0)node[below,color=black]{A};
				\foreach \x in {-2,...,1}
					{
					\draw[densely dotted] (-0.5,1)--(\x,0);
					\draw[fill=sussexg] (\x,0) circle(1mm);
					
					}
					
					\draw[fill=sussexg] (-.5,1) circle(1mm)node[above]{$\cV_0$};
				
				\draw[densely dotted, line width=1.5pt](4,1)--(4.75,0)node[below]{?};
				\draw[color=sussexp, line width=2pt](7,2)--(4,1)--(4.25,0)node[below, color=black]{B}; 
				
				\draw[densely dotted,line width=1.5pt](7,2)--(6,1)--(6.75,0)node[below,color=black]{?};
				\draw[color=sussexp, line width=2pt](7,2)--(6,1)--(5.25,0)node[below,color=black]{C}; 
				\draw[color=sussexp, line width=2pt](7,2)--(6,1)--(5.75,0)node[below,color=black]{D};
				 
				\draw[densely dotted,line width=1.5pt](7,2)--(8,1)--(7.25,0)node[below,color=black]{?}; 
				\draw[color=sussexp, line width=2pt](7,2)--(8,1)--(8.25,0)node[below,color=black]{E};

				\draw[densely dotted,line width=1.5pt](7,2)--(10,1)--(9.75,0)node[below,color=black]{?};

					\draw[densely dotted](4,1) -- (3.25,0); 	
					\draw[densely dotted](4,1) -- (3.75,0); 
					\draw[densely dotted](4,1)--(4.25,0); 
					\draw[densely dotted](4,1) -- (4.75,0); 	
					
					\draw[fill=sussexg](3.25,0)circle(1mm); 	
					\draw[fill=sussexg](3.75,0)circle(1mm); 
					\draw[fill=sussexg](4.25,0)circle(1mm); 
					\draw[fill=sussexg](4.75,0)circle(1mm); 	
					
					\draw[densely dotted](6,1) -- (5.25,0); 	
					\draw[densely dotted](6,1) -- (5.75,0); 
					\draw[densely dotted](6,1)--(6.25,0); 
					\draw[densely dotted](6,1) -- (6.75,0); 	
					
					\draw[fill=sussexg](5.25,0)circle(1mm); 	
					\draw[fill=sussexg](5.75,0)circle(1mm); 
					\draw[fill=sussexg](6.25,0)circle(1mm); 
					\draw[fill=sussexg](6.75,0)circle(1mm); 
					
					\draw[densely dotted](8,1) -- (7.25,0); 	
					\draw[densely dotted](8,1) -- (7.75,0); 
					\draw[densely dotted](8,1)--(8.25,0); 
					\draw[densely dotted](8,1) -- (8.75,0); 	
					
					\draw[fill=sussexg](7.25,0)circle(1mm); 	
					\draw[fill=sussexg](7.75,0)circle(1mm); 
					\draw[fill=sussexg](8.25,0)circle(1mm); 
					\draw[fill=sussexg](8.75,0)circle(1mm); 
					
					\draw[densely dotted](10,1) -- (9.25,0); 	
					\draw[densely dotted](10,1) -- (9.75,0); 
					\draw[densely dotted](10,1)--(10.25,0); 
					\draw[densely dotted](10,1) -- (10.75,0); 	
					
					\draw[fill=sussexg](9.25,0)circle(1mm); 	
					\draw[fill=sussexg](9.75,0)circle(1mm); 
					\draw[fill=sussexg](10.25,0)circle(1mm); 
					\draw[fill=sussexg](10.75,0)circle(1mm);

				\foreach \x in {2,...,5}
					{
					\draw[densely dotted](7,2)--(2*\x,1);
					\draw[fill=sussexg] (2*\x,1) circle(1mm);
					 }

				\draw[fill=sussexg] (7,2) circle(1mm) node[above]{$\cV_1$};

			\end{tikzpicture}
			\hspace{1.5in}
			
			\begin{tikzpicture}[xscale=0.7]
			\definecolor{sussexg}{rgb}{0,0.7,0.7}
			\definecolor{sussexp}{rgb}{0.6,0,0.6} 
				
				\draw[color=sussexp, line width=2pt](10,3)--(4,2)--(2.5,1)--(2.5+-1*0.12,0)node[below,color=black]{\tiny F}; 
				\draw[color=sussexp, line width=2pt](10,3)--(4,2)--(2.5,1)--(2.5+3*0.12,0)node[below,color=black]{\tiny G};
				\draw[color=sussexp, line width=2pt](10,3)--(4,2)--(3.5,1)--(3.5+1*0.12,0)node[below,color=black]{\tiny H}; 
				
				\draw[color=sussexp, line width=2pt](10,3)--(8,2)--(6.5,1)--(6.5+1*0.12,0)node[below,color=black]{\tiny I}; 
				\draw[color=sussexp, line width=2pt](10,3)--(8,2)--(6.5,1)--(6.5+3*0.12,0)node[below,color=black]{\tiny J};
				\draw[color=sussexp, line width=2pt](10,3)--(8,2)--(7.5,1)--(7.5-1*0.12,0)node[below,color=black]{\tiny K}; 
				\draw[color=sussexp, line width=2pt](10,3)--(8,2)--(7.5,1)--(7.5+1*0.12,0)node[below,color=black]{\tiny L}; 
				\draw[color=sussexp, line width=2pt](10,3)--(8,2)--(7.5,1)--(7.5+3*0.12,0)node[below,color=black]{\tiny M}; 
				\draw[color=sussexp, line width=2pt](10,3)--(8,2)--(8.5,1)--(8.5-3*0.12,0)node[below,color=black]{\tiny N}; 
				\draw[color=sussexp, line width=2pt](10,3)--(8,2)--(8.5,1)--(8.5-1*0.12,0)node[below,color=black]{\tiny O}; 
				
				\draw[color=sussexp, line width=2pt](10,3)--(12,2)--(11.5,1)--(11.5+1*0.12,0)node[below,color=black]{\tiny P}; 
				\draw[color=sussexp, line width=2pt](10,3)--(12,2)--(13.5,1)--(13.5-1*0.12,0)node[below,color=black]{\tiny Q}; 
				
				\draw[color=sussexp, line width=2pt](10,3)--(16,2)--(14.5,1)--(14.5-1*0.12,0)node[below,color=black]{\tiny R}; 
				\draw[color=sussexp, line width=2pt](10,3)--(16,2)--(14.5,1)--(14.5+1*0.12,0)node[below,color=black]{\tiny S}; 
				\draw[color=sussexp, line width=2pt](10,3)--(16,2)--(16.5,1)--(16.5-3*0.12,0)node[below,color=black]{\tiny T}; 
				\draw[color=sussexp, line width=2pt](10,3)--(16,2)--(16.5,1)--(16.5-1*0.12,0)node[below,color=black]{\tiny U}; 
				\draw[color=sussexp, line width=2pt](10,3)--(16,2)--(16.5,1)--(16.5+1*0.12,0)node[below,color=black]{\tiny V}; 
				\draw[color=sussexp, line width=2pt](10,3)--(16,2)--(16.5,1)--(16.5+3*0.12,0)node[below,color=black]{\tiny W}; 
				\draw [decorate,decoration={brace, mirror, amplitude = 3pt}]
						(4.5 -3*0.14,-0.1) -- (4.5 +3*0.14,-0.1);
				\draw(4.5,-0.1)node[below]{\tiny ?};	
				
				\draw [decorate,decoration={brace, mirror, amplitude = 3pt}]
						(9.5 -3*0.14,-0.1) -- (9.5 +3*0.14,-0.1);
				\draw(9.5,-0.1)node[below]{\tiny ?};		

				\draw [decorate,decoration={brace, mirror, amplitude = 3pt}]
						(10.5 -3*0.14,-0.1) -- (10.5 +3*0.14,-0.1);
				\draw(10.5,-0.1)node[below]{\tiny ?};	
				
				\draw [decorate,decoration={brace, mirror, amplitude = 3pt}]
						(15.5 -3*0.14,-0.1) -- (15.5 +3*0.14,-0.1);
				\draw(15.5,-0.1)node[below]{\tiny ?};	
								
				\foreach \x in {4,8,12,16}
					{
					\draw[densely dotted] (10,3)--(\x,2);

						\foreach \y in {-3,-1,1,3}{
							\draw[densely dotted] (\x,2) -- (\x + .5*\y, 1);
								
								\foreach \z in  {-3,-1,1,3}{ 
									\draw[densely dotted](\x+.5*\y,1)--(\x+.5*\y+.12*\z,0);
									\draw[fill=sussexg] (\x+.5*\y+.12*\z,0) circle(0.8mm);
								}
							\draw[fill=sussexg] (\x+.5*\y,1) circle(1mm);
							}
						\draw[fill=sussexg] (\x,2) circle(1mm); 	
					 }

				\draw[fill=sussexg](10,3)circle(1mm)node[above]{$\cV_2$};
			\end{tikzpicture}
		\end{center}
	\label{fig:forest}	
	\caption{A forest that corresponds to the percolation cluster 
	of Figure \ref{fig:cubeperc}. The trees correspond to the branching processes 
	in each $\cV_k$ as in Figure \ref{fig:treerep}. 
	The thickset purple lines correspond to the surviving population (i.e., the colored squares) 
	in each shell $\cS_k$. The question marks signify that we do not have 
	information about 
	that branch of the process if we are allowed to look only at the percolation cluster;
	they correspond to lattice squares outside the shell $\cS_k$.}
\end{figure}

Because the directed tree $\mathcal{T}_k$ is finite for every $k\ge 0$, we can isometrically embed
$\mathcal{T}_k$ in $\R^2$ such that the vertices of $\mathcal{T}_k$
that have maximal depth lie on the real axis. Of course,
there are infinitely-many such possible
isometric embeddings; we will choose and fix one [it will not matter which
one]. In this way, we can think of every $\mathcal{T}_k$ as a finite rooted
tree, drawn in $\R^2$, whose vertices of maximal depth lie on the real axis
and whose root lies $k+1$ units above the real axis.
Because every $x\in\Z^d\cap\cS_k$ has been coded by the rays of
$\mathcal{T}_k$, and since those rays can in turn be coded by their last vertex [these are
vertices of maximal depth], thus we obtain a bijection $\pi_k$ that maps
each point $x\in\Z^d\cap\cS_k$
to a certain point $\pi_k(x)$ on the real axis of $\R^2$. Note that,
in this way, $\{\pi_k(x)\}_{x\in\Z^d\cap\cS_k}$ can be identified with a finite
collection of points on the real line.

The collection $\{\mathcal{T}_k\}_{k=0}^\infty$ is a 
forest representation of $\Z^d$. We  use this representation in 
order to impose an order relation $\prec$ on $\Z^d$ as follows:
\begin{enumerate}
	\item If $x\in\Z^d\cap\cS_k$ and $y\in\Z^d\cap\cS_\ell$ 
		for two integers $k,l\ge 0$ such that
		$k< \ell$, then  we declare $x\prec y$;
	\item If $x,y\in\Z^d\cap\cS_k$ for the same integer $k\ge0$, and $\pi_k(x)\le \pi_k(y)$, 
		then we declare $x\prec y$.
\end{enumerate}

It can be checked, using only first principles, 
that $\prec$ is in fact a bona fide total order on $\Z^d$.
We might sometimes  also write $y\succ x$ in place of $x\prec y$.

If we identify $x,y\in\Z^d\cap\cS_k$ with 2 rays of the
tree $\mathcal{T}_k$, viewed as a tree drawn in $\R^d$ as was described earlier,
then $x\prec y$ iff the ray for $x$ lies on, or to the left of, the ray for $y$.
And if $x\in\cS_k$ and $y\in\cS_\ell$ for $k<\ell$, then our definition
of $x\prec y$ stems from 
the fact that we would like to draw the tree $\mathcal{T}_k$ to the left of 
the tree $\mathcal{T}_\ell$,
as we embed the forest $\mathcal{T}_0,\mathcal{T}_1,\ldots$, tree by tree,
isometrically in $\R^2$.

\section{ Martin capacity of fractal percolation}\label{sec:mc}
Now we return to Markov chains.
Throughout this section, let $X:=\{X_n\}_{n=0}^\infty$ denote a transient 
Markov chain on $\Z^d$. 
This chain is constructed in the usual way:
We have a probability space $(A\,,\mathcal{A}\,,P)$
together with a family  $\{P^a\}_{a\in\Z^d}$ of probability measures 
such that under $P^a$,
the Markov chain begins at $X_0=a$ for every $a\in\Z^d$. 
By $E^a$ we mean the corresponding expectation operator for $P^a$
for all $a\in\Z^d$ [$E^a(f):=\int f\,\d P^a$].

We assume that the Markov chain is independent of the fractal 
percolations. The two processes are jointly constructed on
$(A\times\Omega\,,\mathcal{A}\times\mathcal{F},P\times\P)$. 
We write $\mathbb{P}^a:=P^a\times\P$ and $\mathbb{E}^a$ the corresponding expectation
operator [$\mathbb{E}^a(f):=\int f\,\d\mathbb{P}^a$].

As is well known, 
the transience of $X$ is equivalent to seemingly-simpler condition that
$g(x\,,x)<\infty$ for all $x\in\Z^d$. This is because 
$g(a\,,x)\le g(x\,,x)$ for all $x\in\Z^d$; in fact, one can apply the strong Markov property
to the first hitting time of $x\in\Z^d$ in order to see that
\begin{equation}\label{HitPoints}
	g(a\,,x) = g(x\,,x)\cdot P^a\{ X_n=x\text{ for some $n\ge 0$}\}
	\qquad\text{for every $x\in\Z^d$}.
\end{equation}

We define an equivalence relation on $\Z^d$ as follows:
For all $x,y\in\Z^d$ we write ``$x\leftrightarrow  y$'' when there
exists an integer $k\ge 0$ such that $x$ and $y$ are both in
the same shell $\cS_k$. Symbol $x\not\leftrightarrow  y$ denotes that $x$ and $y$ are in different shells.

If $\mu$ is a probability measure on $\Z^d$, then we
define two ``energy forms'' for $\mu$.  The first form is defined,
for every fixed $a\in\Z^d$, as 
\begin{equation}\begin{split}
	I (\mu\,;a) &:= \mathop{\sum\sum}_{\substack{
	x,y\in\Z^d:\\ x\not\leftrightarrow  y}}\frac{g(x\,,y)}{g(a\,,y)} \mu(x)\,\mu(y),
\end{split}\end{equation}
where $\mu(z) := \mu(\{z\})$ as before. 
Our second definition of energy requires an additional
parameter $p\in(0\,,1]$, and is defined as follows:
\begin{equation}
	I_p(\mu\,;a) := \mathop{\sum\sum}\limits_{\substack{x,y\in\Z^d:\\
	x\leftrightarrow  y}} 
	\frac{g(x\,,y)}{g(a\,,y)} p^{-{\rm d}(x,y)}\mu(x)\mu(y).
\end{equation}
Note, in particular, that 
\begin{equation}
	I(\mu\,;a) + I_1(\mu\,;a) = \mathop{\sum\sum}_{x,y\in\Z^d}
	\frac{g(x\,,y)}{g(a\,,y)}\,\mu(x)\,\mu(y)
\end{equation}
coincides with the \emph{Martin energy of $\mu$} \cite{BPP1995}. 

Finally we define a quantity that can be thought of as 
a kind of graded Martin capacity
associated to $X$: For any set $F$, $p \le 1$ and $a \in \Z^d$, define 
the Martin $p$-capacity by
\begin{equation}\label{c:k:p}
		c_p(F;a) := \sup_{\substack{F_0\subseteq F:\\
		F_0\text{ \rm finite}}}\Big[\inf_{\mu\in M_1(F_0)}\big\{
		I(\mu\,;a) + I_p(\mu\,;a)\big\} \Big]^{-1}.
	\end{equation}
The set function $c_1$ is the same Martin capacity
that appeared earlier in \eqref{c1}.
It might help to observe that the Martin $p$-capacity satisfies the following monotonicity property:  
\begin{equation}\label{eq:cp:mono}
	\textrm{If $F \subseteq G$ \quad then \quad }c_p(F;a) \le c_p(G;a). 
\end{equation}

The main result of this section can
be stated as follows.

\begin{theorem}\label{th:RW:FP}
	If $F\subseteq\Z^d$ is nonrandom, then for all $a\in\Z^d$,
	\begin{equation}
		\tfrac12 c_p(F;a)\le
		\mathbb{P}^a\{X_n\in\Pi_p\cap F\text{ \rm for some $n\ge 0$}\}
		\le 128 c_p(F;a),
	\end{equation}
	where $c_p$ is defined by \eqref{c:k:p}.
\end{theorem}

Theorem \ref{th:RW:FP}  implies the following.

\begin{corollary}\label{cor:main}
	If $F\subseteq\Z^d$ is recurrent for $X$ $P^a$-a.s.,
	and $a:=X_0\in\Z^d$ and $F$ are  nonrandom, then
	\begin{equation}
		\Dim(\mathcal{R}_X\cap F) = -\log_2 p_c(F\,;a)\qquad\text{$P^a$-a.s.},
	\end{equation}
	where 
	\begin{equation}
		p_c(F\,;a) :={\inf} \bigg\{ p\in[2^{-d},1]:{\inf_{\substack{G\subset\Z^d:\\
		G\text{ \rm is cofinite}}}}\ c_p(F\cap G\,;a)>0 \bigg\},
	\end{equation}
	and $\{c_p(\bullet\,;a)\}_{p\le 1}$ is defined in \eqref{c:k:p}
	above.
\end{corollary}

\begin{proof}
	Suppose that there exists a number $p\in[2^{-d}\,,1]$ for which
	\begin{equation}
		\tau(p):=\tfrac12\inf c_p(F\cap G\,;a)>0, 
	\end{equation}
	where the infimum is computed over all cofinite
	sets $G\subset\Z^d$. Define
	\begin{equation}
		G_N:= \left\{x\in\Z^d:\, \|x\|>N\right\}
		\qquad\text{for all $N\ge1$.}
	\end{equation}
	According to Theorem \ref{th:RW:FP},
	\begin{equation}
		\inf_{N\ge 1}\mathbb{P}^a\left\{ \mathcal{R}_X\cap\Pi_p\cap
		F\cap G_N\neq\varnothing\right\}
		\ge\tau(p)>0.
	\end{equation}
	It follows from this and elementary properties of probabilities that
	\begin{equation}
		\mathbb{P}^a\left\{ \mathcal{R}_X\cap\Pi_p\cap
		F\cap G_N\neq\varnothing\text{ for infinitely many $N\ge1$}\right\}
		\ge\tau(p)>0.
	\end{equation}
	This in turn shows that
	\begin{equation}
		\mathbb{P}^a\{ \mathcal{R}_X\cap F\Recurrent\Pi_p\}
		\ge\tau(p)>0.
	\end{equation}
	Apply Lemma \ref{lem:0-1}, conditionally on
	$\mathcal{R}_X$, in order to deduce from the preceding that
	$\mathcal{R}_X\cap F\Recurrent \Pi_p$ a.s.\ $[\mathbb{P}^a]$. In particular,
	we apply Theorem \ref{th:Dim:Pi}, once again conditionally on $\mathcal{R}_X$,
	in order to see that
	\begin{equation}
		\Dim(\mathcal{R}_X\cap F)\ge-\log_2p\qquad\mathbb{P}^a\text{-a.s.}
	\end{equation}
	Optimize over
	our choice of $p$ to see that 
	\begin{equation}
		\Dim(\mathcal{R}_X\cap F)\ge
		-\log_2 p_c(F\,;a)\qquad\mathbb{P}^a\text{-a.s.}
	\end{equation}
	
	For the other bound, it is enough to consider the case that 
	$p_c\in  (2^{-d}\,,1]$ since $p_c=2^{-d}$ yields a trivial bound.
	Thus,  we can consider instead a number $p\in[2^{-d}\,,1)$ such that
	$\inf c_p(F\cap G\,;a)=0$, where once again the infimum is over all cofinite
	sets $G\subset\Z^d$. It is easy to deduce from this choice of $p$ and
	Theorem \ref{th:RW:FP} that
	\begin{equation}
		\lim_{N\to\infty}\mathbb{P}^a\{ \mathcal{R}_X\cap\Pi_p\cap
		F\cap G_N\neq\varnothing\}=0.
	\end{equation}
	Since the random walk $X$ is transient,
	and because $\mathcal{R}_X\cap F$ is a.s.\ recurrent,
	elementary properties of probabilities imply that
	$\mathbb{P}^a\{\mathcal{R}_X\cap F\Recurrent\Pi_p\}=0$.
	Therefore, we may apply Theorem \ref{th:Dim:Pi}, one more time conditionally 
	on $\mathcal{R}_X$,
	in order to see that $\Dim(\mathcal{R}_X\cap F)\le-\log_2p$. Optimize
	over our choice of $p$ to see that
	\begin{equation}
		\Dim(\mathcal{R}_X\cap F)\le -\log_2 p_c(F\,;a)
		\qquad\mathbb{P}^a\text{-a.s.}
	\end{equation}
	The corollary follows.
\end{proof}

Theorem \ref{th:RW:FP} has a number of other consequences as well.
The following is
a universal estimate on the expected Martin $p$-capacity of
the range of the fractal percolation set in a shell $\cS_k$.

\begin{corollary}\label{cor:MC1}
	For every point $a\in\Z^d$, integers $k\ge 0$,
	nonrandom finite sets $F\subset\Z^d$, and 
	percolation parameters $p,q\in[2^{-d}\,,1]$,
	\begin{equation}
		(256)^{-1} c_{pq}(F;a)\le\E\left[ c_p(\Pi_q\cap F;a)\right] \le 256 c_{pq}(F;a).
	\end{equation}
\end{corollary}

\begin{proof}[Proof of Corollary \ref{cor:MC1}]
	Let $\Pi_q'$ denote an independent copy of $\Pi_q$
	and denote the corresponding (independent of $\mathbb{P}^a$) measure by $\P'$
	with corresponding expectation operator $\E'$. Theorem \ref{th:RW:FP} ensures that
	\begin{equation}
		\mathbb{P}^a\{ \mathcal R_X \cap \Pi_p\cap\Pi_q'\cap F\neq\varnothing\} 
		\le 128 c_p(\Pi_q'\cap F\,;a)\qquad\text{$\P'$-a.s.}
	\end{equation}
	Therefore we integrate $[\P']$ in order to see that
	\begin{equation}\label{eq:b1}
		(\mathbb{P}^a\times\P')\{ \mathcal R_X \cap \Pi_p\cap\Pi_q'\cap F\neq\varnothing\} 
		\le128  \E[c_p(\Pi_q\cap F;a)]. 
	\end{equation}
	
	For the other bound, we recall that $\Pi_p\cap\Pi_q'$ has
	the same law [$\P\times\P'$] as $\Pi_{pq}$ does $[\P]$. Therefore,
	Theorem \ref{th:RW:FP} implies that
	\begin{equation}\label{eq:b2}\begin{split}
		(\mathbb{P}^a \times \P')\{ 
			\mathcal R_X\cap \Pi_p\cap\Pi_q'\cap F\neq\varnothing\}
			&=\mathbb{P}^a\{\mathcal R_X\cap \Pi_{pq} \cap F\neq\varnothing \}\\
		&\ge \tfrac{1}{2}c_{pq}(F;a). 
	\end{split}\end{equation}
	Together, \eqref{eq:b1} and \eqref{eq:b2} yield
	$c_{pq}(F; a) \le 256 \E[c_p(\Pi_q\cap F;a)].$
	 The other bound in the statement follows similarly.
\end{proof}

The second consequence of Theorem
\ref{th:RW:FP} is a Lamperti-type condition on recurrence
that was stated in Corollary \ref{cor:Lamperti}.

\begin{proof}[Proof of Corollary \ref{cor:Lamperti}]
	Consider the stopping times $\{T_k\}_{k=0}^\infty$
	defined by
	\begin{equation}
		T_k := \inf\left\{ n\ge 0:\, X_n\in F\cap\cS_k\right\}
		\qquad\text{for all $k\ge 0$},
	\end{equation}
	where $\inf\varnothing:=\infty$. Theorem \ref{th:RW:FP}
	ensures that
	\begin{equation}\label{eq:BPP}
		P^z\{T_m<\infty\} \asymp c_1(F\cap\cS_m;z)
	\end{equation}
	for all integers $m$, and $z\in\Z^d$.
	
	Of course, $F\Recurrent\mathcal{R}_X$
	if and only if $\sum_{k=0}^\infty \1_{\{T_k<\infty\}}=\infty$.
	Therefore, if $\sum_{k=0}^\infty
	c_1(F\cap\cS_k;a)<\infty$, then the Borel--Cantelli lemma and
	\eqref{eq:BPP} together imply
	that $\mathcal{R}_X\nRecurrent F$. Note that this portion does
	not require the Lamperti condition \eqref{cond:Lamperti}.
	The complementary half of the corollary does.

	If, on the other hand, $\sum_{k=0}^\infty c_1(F\cap \cS_k;a)=\infty$,
	then \eqref{eq:BPP} ensures that
	$\sum_{k=1}^\infty P^a\{T_k<\infty\}=\infty$. A standard
	second moment argument reduces our problem to showing 
	the existence of a positive constant $C_0$
	such that
	\begin{equation}\label{goal:BC}
		\sum_{k=0}^N\sum_{j=k}^N P^a\{T_k<\infty, T_j<\infty\}
		\le C_0 \bigg( \sum_{k=0}^N P^a\{T_k<\infty\}\bigg)^2,
	\end{equation}
	as $N\to\infty$. This is what we aim to prove.

	By the strong Markov property,
	\begin{equation}\begin{split}
		P^a\{T_k<T_{k+l}<\infty\} & = E^a\left[ \1_{\{T_k<\infty\}}
			P^{X_{T_k}}\{ T_{k+l}<\infty\}\right]\\
		&\le 128 E^a\left[ \1_{\{T_k<\infty\}}
			c_1(F\cap\cS_{k+l};X_{T_k})\right]\\
		&\le 128 P^a\{T_k<\infty\}\sup_{z\in\Z^d\cap\cS_k} c_1(F\cap\cS_{k+l};z).
	\end{split}\end{equation}
	Let us observe that, thanks to \eqref{cond:Lamperti}, there exists
	a finite constant $K>1$ such that $g(x\,,y)\le K g(a\,,y)$
	whenever $x\in\cS_n$ and $y\in\cS_m$ for integers $m>n\ge K$
	such that $m\ge n+K$. Thus, it follows readily from the
	definition \eqref{c:k:p} of $c_1$ that
	\begin{equation}
		\max_{z\in\cS_k\cap \Z^d}c_1(F\cap\cS_{k+l};z) \le K
		c_1(F\cap\cS_{k+l};a), 
	\end{equation}
	uniformly for all integers $k,l\ge K$. 
	In accord with \eqref{eq:BPP},
	\begin{equation}
		P^a\{T_k<T_{k+l}<\infty\}\le 256K P^a\{T_k<\infty\}P^a\{T_{k+l}<\infty\},
	\end{equation}
	whenever $k,l\ge K$.
	
	Similarly, we can appeal to \eqref{cond:Lamperti} in order
	to find a finite constant $K'>1$ such that
	\begin{equation}
		P^a\{T_{k+l}<T_k<\infty\} 
		\le 256K' P^a\{T_{k+l}<\infty\}P^a\{T_k<\infty\},
	\end{equation}
	as long as $k,l\ge K'$. Let $K_0:=256\max(K\,,K')$.
	Because we can write $P^a\{ T_k<\infty,T_{k+l}<\infty\}$ as
	$P^a\{T_k < T_{k+l}<\infty\} + P^a\{T_{k+l}<T_k<\infty\}$,
	the preceding two displayed bounds together imply that
	\begin{align}\notag
		&\sum_{k=K_0}^N\sum_{j=k}^N P^a\{ T_k<\infty, T_j<\infty\}\\\notag
		&\le\sum_{k=K_0}^N\sum_{j=k+K_0}^N P^a\{ T_k<\infty, T_j<\infty\}
			+ \sum_{k=K_0}^N\sum_{j=k}^{k+K_0-1} P^a\{ T_k<\infty, T_j<\infty\}\\
		&\le K_0\bigg(\sum_{k=0}^N P^a\{T_k<\infty\}\bigg)^2
			+ K_0\sum_{k=0}^N P^a\{T_k<\infty\}.
	\end{align}
	Since $\sum_{k=0}^N P^a\{T_k<\infty\}\to \infty$ as $N\to\infty$,
	we obtain \eqref{goal:BC} if $C_0 = 2K_0$ for all $N$ large.  The corollary follows immediately.
	\end{proof}

Let us mention a final corollary of Theorem
\ref{th:RW:FP}. That corollary presents a more tractable 
formula for $\Dim(\mathcal{R}_X\cap F)$,
valid under the Lamperti-type condition
\eqref{cond:Lamperti}.

\begin{corollary}\label{cor:RW:Dim}
	Let $F\subset\Z^d$ and $X_0:=a\in\Z^d$ be non random.
	Then, $\Dim(\mathcal{R}_X\cap F)\le -\log_2 p_c(F\,;a)$ a.s.\ $[P^a]$,
	where
	\begin{align}\label{def:p_c}
		p_c(F\,;a) :&={\sup}\Big\{ p\in[2^{-d}\,,1]:\ 
			\sum_{k=0}^\infty c_p(F\cap\cS_k;a)<\infty\Big\}\\
		&={\inf\Big\{ p\in[2^{-d}\,,1]:\ \sum_{k=0}^\infty
			c_p(F\cap\cS_k;a)=\infty\Big\}.}
	\end{align}
	If, in addition, \eqref{cond:Lamperti} holds, then
	\begin{equation}
		\Dim(\mathcal{R}_X\cap F)=-\log_2p_c(F\,;a) 
		\qquad\text{a.s.\  $[\mathbb{P}^a]$.}
	\end{equation}
\end{corollary}

\begin{remark}
	If $X$ is a random walk that satisfies the conditions of
	Proposition \ref{pr:BT} and starts at $0$, 
	then our previous comments in 
	Example \ref{ex:RW:nice} imply that 
	\begin{equation}\begin{split}
		c_p(F\cap\cS_k\,;0) \asymp 2^{k\alpha}
		\mathop{\sum\sum}_{x,y\in\cS_k} g(x\,,y)p^{-\d(x,y)}\,\mu(x)\,\mu(y).
	\end{split}\end{equation}
\end{remark}

\begin{proof}[Proof of Corollary \ref{cor:RW:Dim}]
	First we prove the upper bound on $\Dim(\mathcal{R}_X\cap F)$.
	
	If $\Dim(\mathcal{R}_X\cap F)>-\log_2 p$ for
	some $p\in(2^{-d}\,,1]$, then Theorem \ref{th:Dim:Pi} ensures that
	$X\Recurrent\Pi_p$; see especially Remark \ref{rem:Dim:Pi}.
	This, the easy half of the Borel--Cantelli lemma, and Theorem
	\ref{th:RW:FP}  together imply that 
	$\sum_{k=0}^\infty c_p(F\cap\cS_k;a)=\infty$. Optimize over
	$p\in(2^{-d}\,,1]$ in order to deduce that 
	\begin{equation}
		\Dim(\mathcal{R}_X\cap F)\le -\log_2 p_c(F\,;a)
		\qquad\mathbb{P}^a\text{-a.s.}
	\end{equation}
	
	In the reverse direction we assume that \eqref{cond:Lamperti} holds,
	and strive to show that
	\begin{equation}\label{Goal:LB}
		\Dim(\mathcal{R}_X\cap F)\ge-\log_2 p_c(F\,;a)
		\qquad\mathbb{P}^a\text{-a.s.}
	\end{equation}
	There is nothing to prove if $p_c(F\,;a)=1$. Therefore,
	we assume without loss of generality that
	\begin{equation}
		2^{-d} \le p_c(F\,;a) <1.
	\end{equation}
	
	According to Theorem \ref{th:RW:FP}, and thanks to 
	the definition of the critical probability $p_c(F\,;a)$,
	$\sum_{k=0}^\infty \mathbb{P}^a\{\mathcal{R}_X\cap F\cap\Pi_p\cap \cS_k\neq\varnothing\}
	=\infty$ for every $p\in(p_c(F\,;a)\,,1]$. That is,
	\begin{equation}\label{eq:tau:N}
		\lim_{N\to\infty}\mathbb{E}^a(\tau_N)=\infty,
		\text{ where }\tau_N:=\sum_{k=0}^N\1\{\mathcal{R}_X\cap
		F\cap \Pi_p\cap\cS_k
		\neq\varnothing\}.
	\end{equation}
	Next we verify that there exists a uniform positive constant $A$ so that
	\begin{equation}\label{goal:Et^2}
		\mathbb{E}^a\left(\tau_N^2\right) \le A \big( \mathbb{E}^a(\tau_N)\big)^2
		\qquad\text{as $N\to\infty$}.
	\end{equation}
	
	By the Markov property of $X$ and
	the particular construction of $\Pi_p$,
	\begin{align}
		&\mathbb{E}^a\left(\tau_N^2\right)\\\notag
		&\le 2
			\mathop{\sum\sum}\limits_{0\le j\le k\le N}
			\mathbb{P}^a\left\{ \mathcal{R}_X\cap F\cap \Pi_p\cap\cS_j
			\neq\varnothing\,,
			\mathcal{R}_X\cap F\cap\Pi_p\cap\cS_k
			\neq\varnothing\right\}\\\notag
		&\le 2
			\mathop{\sum\sum}\limits_{0\le j\le k\le N}
			\mathbb{P}^a\left\{ \mathcal{R}_X\cap F\cap\Pi_p\cap\cS_j
			\neq\varnothing\right\}\max_{z\in\Z^d\cap\cS_j}\mathbb{P}^z\left\{
			\mathcal{R}_X\cap F\cap\Pi_p\cap\cS_k
			\neq\varnothing\right\}.
	\end{align}
	Theorem \ref{th:RW:FP} then implies that
	\begin{equation}
		\mathbb{E}^a\left(\tau_N^2\right)\le C
		\mathop{\sum\sum}\limits_{0\le j\le k\le N}
		c_p(F\cap\cS_j;a)\max_{z\in\Z^d\cap\cS_j}
		c_p(F\cap \cS_k;z),
	\end{equation}
	where $C:=32768$. We now apply an argument very similar to the
	one used to produce \eqref{goal:BC} in order to see that there exists
	an integer $K_*>1$ such that
	\begin{equation}
		\max_{z\in\Z^d\cap
		\cS_j}c_p(F\cap \cS_k;z)\le K_* c_p(F\cap \cS_k;a),
	\end{equation}
	as long as $k\ge j+K_*$. In this way we find that
	\begin{align}
		&\mathbb{E}^a\left(\tau_N^2\right)\\\notag
		&\le CK_*\bigg(\sum_{j=0}^N c_p(F\cap \cS_k;a)\bigg)^2
			+ \mathop{\sum\sum}\limits_{\substack{0\le j\le k\le N:\\
			k< j+K_*}} c_p(F\cap\cS_j;a)\max_{z\in\Z^d\cap\cS_j}
			c_p(F\cap \cS_k;z).
	\end{align}
	Since $\sup_{z\in\Z^d} c_p(F\cap\cS_k;z)\le 2$ [see Theorem \ref{th:RW:FP}],
	it follows that
	\begin{equation}
		\mathbb{E}^a\left(\tau_N^2\right)
		\le CK_*\bigg(\sum_{j=0}^N c_p(F\cap \cS_k;a)\bigg)^2
		+ 2K_*\sum_{j=0}^N c_p(F\cap\cS_j;a).
	\end{equation}
	Therefore, Theorem \ref{th:RW:FP} shows that
	\begin{equation}
		\mathbb{E}^a(\tau_N^2)\le 
		4CK_*[\mathbb{E}^a(\tau_N)]^2+ 2K_*\mathbb{E}^a(\tau_N).
	\end{equation}
	Because of the 0-1 law [see Lemma \ref{lem:0-1}],
	this and \eqref{eq:tau:N} together imply that $\tau_N\to\infty$
	a.s.\ $[\mathbb{P}^a]$ as $N\to\infty$. This is another way to state that
	$\mathcal{R}_X\cap F\Recurrent\Pi_p$ a.s.\ $[\mathbb{P}^a]$.
	Theorem \ref{th:Dim:Pi}---see, in particular, Remark \ref{rem:Dim:Pi}---then
	implies that
	\begin{equation}
		\Dim(\mathcal{R}_X\cap F)\ge -\log_2 p\qquad
		\mathbb{P}^a\text{-a.s.}
	\end{equation}
	Since $p\in(p_c(F\,;a)\,,1]$ were arbitrary, the lower bound \eqref{Goal:LB} follows.
\end{proof}

\begin{proof}[Proof of Theorem \ref{th:RW:FP}]
	Because $\mathbb{P}^a\{X_n\in\Pi_p\cap F\text{ \rm for some $n\ge 0$}\}$
	is equal to 
	\begin{equation}
		\sup_{\substack{F_0\subseteq F:\\
		F_0\text{ is finite}}}
		\mathbb{P}^a\{X_n\in\Pi_p\cap F_0\text{ \rm for some $n\ge 0$}\},
	\end{equation}
	we can assume without loss of generality that $F$ is a finite set.
	
	The first inequality of the proposition
	follows readily by adapting the second-moment argument
	of Benjamini et al \cite{BPP1995}. The few details follow.
	
	For every $x=(x_1, \dots, x_d)\in \Z^d$, there exists a unique
	positive integer $k$ such that $x \in \cS_k$. Let $\Delta(x) := k+1$
	for this pairing of $x\in\Z^d$ and $k\ge 1$.
	Then we define, for all probability measures $\mu\in M_1(F)$, a
	nonnegative random variable
	\begin{equation}
		J_{\mu} := 
		\sum_{n=0}^\infty \frac{\mu(X_n)}{g(a\,,X_n)} 
		\frac{\1_{\{X_n\in\Pi_p\}}}{p^{\Delta(X_n)}},
	\end{equation}
	where $\mu(w) := \mu(\{w\})$ for every $w\in\Z^d$.
	The preceding display contains an almost surely well-defined sum 
	because the summands are non negative
	and $\mu(X_n)/g(a\,,X_n)\le1$ a.s.\ [$P^a$] for all $n\ge 0$.
	We can therefore rearrange the sum and write 
	\begin{equation}\label{J}
		J_{\mu} = \sum_{n=0}^\infty\sum_{x\in\Z^d} \frac{\1_{\{x\}}(X_n)}{g(a\,,x)}
		\frac{\1_{\Pi_p}(x)}{p^{\Delta(x)}}\mu(x).
	\end{equation}
	Because
	 $\P\{x\in\Pi_p\}= p^{\Delta(x)}$ for all $x\in\Z^d$,
	\begin{equation}
		\mathbb{E}^a(J_{\mu})=1.
	\end{equation}
	Similarly, we compute 
	\begin{align}\notag
		\mathbb{E}^a(J_{\mu}^2) &\le 2\sum_{x,y\in\Z^d}
			\mathop{\sum\sum}\limits_{m\ge n\ge 0} \frac{P^a
			\{X_n=x\,,X_m=y\}}{g(a\,,x)g(a\,,y)}
			\frac{\P\left\{x,y\in\Pi_p\right\}}{p^{\Delta(x)+\Delta(y)}}\mu(x)\mu(y)\\
		&=2( I(\mu\,;a)+I_p(\mu\,;a)).
			\label{EJ^2}
	\end{align}
	
	If $J_{\mu}>0$ for some $\mu\in M_1(F)$, then certainly 
	$X_n\in\Pi_p\cap F$ for some $n\ge 0$. Therefore, the
	Paley--Zygmund inequality implies that
	for every $\mu\in M_1(F)$,
	\begin{align}\label{prec1}
		\mathbb{P}^a\{X_n\in\Pi_p\cap F\text{ for some $n\ge 0$}\} &\ge
			\mathbb{P}^a\{J_{\mu}>0\}\ge \frac{[\mathbb{E}^a(J_{\mu})]^2}{\mathbb{E}^a(J_{\mu}^2)}\\\notag
		&\ge \left[ 2\left\{ I(\mu\,;a)+I_p(\mu\,;a)\right\} \right]^{-1}.
	\end{align}
	The left-most quantity does not depend on $\mu\in M_1(F)$; therefore,
	we may optimize the  right-most quantity in \eqref{prec1}
	over all probability measures  $\mu\in M_1(F)$ in order to see that
	$\mathbb{P}^a\{X_n\in\Pi_p\cap F\text{ for some $n\ge 0$}\}
	\ge\frac12 c_p(F;a)$. This is the desired lower bound on the hitting probability
	of the theorem.
	
	Next we verify the complementary probability, still assuming without
	loss of generality that $F$ is finite; that is,
	$F\subseteq\cV_k$ for a nonnegative integer $k$ that
	is still held fixed throughout. Without loss of generality, we may also assume
	\begin{equation}\label{assume}
		\mathbb{P}^a\{X_m\in\Pi_p\cap F\text{ for some $m\ge 0$}\}>0.
	\end{equation}
	Otherwise, there is nothing to prove.
	
	In order to establish the more interesting
	second inequality of the theorem we will need
	to introduce some notation.
	Let $\mathcal{X}_n$ denote the sigma-algebra
	generated by $X_0,\ldots,X_n$ for all $n\ge 0$. 
	
	Recall that, because of our forest representation of
	$\Z^d$, we identify every point $\rho\in\cS_k\cap\Z^d$
	with a ray in a finite tree $\mathcal{T}_k$, which was
	described in \S\ref{sec:forest}. Recall also that $\mathcal{T}_k$
	has been embedded in $\R^2$ so that its deepest vertices lie
	on the real axis of $\R^2$. In this way, we can identify every
	point $\rho\in\cS_k\cap\Z^d$ with a point, which we continue
	to write as $\rho$, on the real axis.
	
	For every $\rho\in\cV_k\cap\Z^d$,
	let $\mathcal{P}_\rho$ denote the sigma-algebra generated by
	the fractal-percolation weights
	$I_p(Q_0(y)),\ldots,I_p(Q_{\Delta(y)+1}(y))$ for all $y\in \Z^d\cap\cV_k$ 
	such that $y\prec \rho$. Similarly, let
	$\mathcal{F}_\rho$ the sigma-algebra generated by
	all of the fractal-percolation weights
	$I_p(Q_0(y)),\ldots,I_p(Q_{\Delta(y)+1}(y))$, where $y\in \Z^d\cap\cV_k$
	satisfies $y\succ \rho$. If we think of $\rho$ as a maximum-depth
	vertex of $\mathcal{T}_k$ and the latter is embedded in $\R^2$,
	as was mentioned earlier, then we can think of $\mathcal{P}_\rho$ as the information, on
	the fractal percolation, on $\mathcal{T}_k$ that lies to the left of
	$\rho$ [including $\rho$]; this is the ``$\mathcal{P}$ast'' of $\rho$. 
	Similarly, we may think of
	$\mathcal{F}_\rho$ as the information to the right of $\rho$;
	this is the ``$\mathcal{F}$uture'' of $\rho$.
	
	Next we define two ``2-parameter martingales,'' $\Lambda$ and
	$V$ as follows:
	\begin{equation}
		\Lambda_{m,\rho} :=\mathbb{E}^a[J_{\mu}\,|\,\mathcal{X}_m\vee
		\mathcal{F}_\rho];\qquad
		V_{m,\rho} :=\mathbb{E}^a[J_{\mu}\,|\,\mathcal{X}_m\vee
		\mathcal{P}_\rho],
	\end{equation}
	for all $m\ge 0$ and $\rho\in\cV_k\cap\Z^d$.
	Because our random walk is independent of the fractal
	percolation, we may write the following
	after we appeal to independence:
	\begin{equation}\begin{split}
		&\Lambda_{m,\rho}\\
		&\ge \sum_{n=m}^\infty\sum_{\substack{x\in\Z^d:\\
			x\succ\rho}}
			\frac{P^a( X_n=x\,|\, \mathcal{X}_m)}{g(a\,,x)}
			\frac{\P(x\in\Pi_p\,|\,\mathcal{F}_\rho)}{p^{\Delta(x)}}\mu(x)\cdot
			\1_{\{X_m=\rho\in\Pi_p\cap F\}}.
	\end{split}\end{equation}
	The Markov property implies the a.s.-inequality,
	\begin{equation}\begin{split}
		\Lambda_{m,\rho} 
			&\ge \sum_{\substack{x\in\Z^d:\\
			x\succ\rho\\
			x\leftrightarrow  \rho}}
			\frac{g(X_m\,,x)}{g(a\,,x)}p^{-{\rm d}(x,X_m)}\mu(x)\cdot
			\1_{\{X_m=\rho\in\Pi_p\cap F\}}\\\notag
		&\hskip1.6in + \sum_{\substack{x\in\Z^d:\\
			x\succ\rho\\
			x\not\leftrightarrow  \rho}}
			\frac{g(X_m\,,x)}{g(a\,,x)}\mu(x)\cdot
			\1_{\{X_m=\rho\in\Pi_p\cap F\}}.
	\end{split}\end{equation}
	We stress, once again, that the ratio of the Green's functions are well defined
	$P^a$-a.s.
	
	Similarly,
	\begin{align}\notag
		V_{m,\rho} 
			&\ge \sum_{n=m}^\infty\sum_{\substack{x\in\Z^d:\\
			x\prec\rho}}
			\frac{P^a( X_n=x\,|\, \mathcal{X}_m)}{g(a\,,x)}
			\frac{\P(x\in\Pi_p\,|\,\mathcal{P}_\rho)}{p^{\Delta(x)}}\mu(x)\cdot
			\1_{\{X_m=\rho\in\Pi_p\cap F\}}\\
		&= \sum_{\substack{x\in\Z^d:\\
			x\prec\rho\\x\leftrightarrow \rho}}
			\frac{g(X_m\,,x)}{g(a\,,x)}p^{-{\rm d}(x,X_m)}\,\mu(x)\cdot
			\1_{\{X_m=\rho\in\Pi_p\cap F\}}\\\notag
		&\hskip1.6in + \sum_{\substack{x\in\Z^d:\\
			x\prec\rho\\x\not\leftrightarrow \rho}}
			\frac{g(X_m\,,x)}{g(a\,,x)}\,\mu(x)\cdot
			\1_{\{X_m=\rho\in\Pi_p\cap F\}}.
	\end{align}
	Therefore, with probability one,
	\begin{equation}\begin{split}
		\Lambda_{m,\rho} + V_{m,\rho} 
			&\ge \sum_{\substack{x\in\Z^d:\\x\leftrightarrow \rho}}
			\frac{g(X_m\,,x)}{g(a\,,x)}p^{-{\rm d}(x,X_m)}\mu(x)\cdot
			\1_{\{X_m=\rho\in\Pi_p\cap F\}}\\
		&\hskip.5in + \sum_{\substack{x\in\Z^d:\\x\not\leftrightarrow \rho}}
			\frac{g(X_m\,,x)}{g(a\,,x)}\mu(x)\cdot
			\1_{\{X_m=\rho\in\Pi_p\cap F\}}.
	\end{split}\end{equation}
	There are only a countable number of such pairs $(m\,,\rho)$.
	Therefore, the previous lower bound on $\Lambda_{m,\rho}$
	holds, off a single null set, simultaneously for all integers
	$m\ge 0$ and integral points $\rho\in\cV_k\cap\Z^d$.
	
	In order to simplify the typesetting, let us write
	\begin{equation}
		\Lambda_*:=\sup_{m\ge 0,\rho\in\cV_k\cap\Z^d}\Lambda_{m,\rho},\qquad
		V_*:=\sup_{m\ge 0,\rho\in\cV_k\cap\Z^d}V_{m,\rho}.
	\end{equation}
	We might note that, with probability one,
	\begin{equation}\label{eq:Lambda:V}\begin{split}
		\Lambda_*+ V_*
			&\ge \sum_{\substack{x\in\Z^d:\\x\leftrightarrow  X_m}}
			\frac{g(X_m\,,x)}{g(a\,,x)}p^{-{\rm d}(x,X_m)}\mu(x)\cdot
			\1_{\{X_m\in\Pi_p\cap F\}}\\
		&\hskip.5in + \sum_{\substack{x\in\Z^d:\\x\not\leftrightarrow  X_m}}
			\frac{g(X_m\,,x)}{g(a\,,x)}\mu(x)\cdot
			\1_{\{X_m\in\Pi_p\cap F\}},
	\end{split}\end{equation}
	simultaneously for all integers $m\ge 0$.
	
	Now we apply an important idea that is,
	in a different form due to Fitzsimmons and Salisbury \cite{FS1989}.
	Define  a $\Z_+\cup\{\infty\}$-valued random variable $M$ by
		\[
		M:=\inf\{m\ge0:\, X_m\in \Pi_p\cap F\},
		\] 
	where $\inf\varnothing:=\infty$.
	$M$ is a stopping time with respect to the filtration of the random walk,
	conditionally on the entire history of the fractal percolation,
	$\P$-a.s.\ on $\{\Pi_p\cap F\neq\varnothing\}$. 
	
	Consider the event,
	\begin{equation}\label{G}
		G := \{ \omega\in\Omega:\ M(\omega)<\infty\,, 
		\Pi_p(\omega)\cap F\neq\varnothing\}.
	\end{equation}
	Hypothesis \eqref{assume} is  another way to state $\mathbb{P}^a(G)>0$.
	Moreover, \eqref{eq:Lambda:V} implies the following
	key a.s.\ inequality:
	\begin{equation}
		\Lambda_* + V_*
		\ge \sum_{x\in\Z^d}
		\frac{g(X_M\,,x)}{g(a\,,x)}\left\{p^{-{\rm d}(x,X_M)}
		\1_{\{x\leftrightarrow  X_M\}}
		+\1_{\{x\not\leftrightarrow  X_M\}}\right\}\mu(x)\cdot
		\1_G.
	\end{equation}
	The preceding is valid $\mathbb{P}^a$-a.s.\ for any probability measure $\mu$ on $F$.
	We apply it using the following particular choice:
	\begin{equation}\label{mu}
		\mu(x) := \mathbb{P}^a( X_M=x\,|\ G)\qquad(x\in\Z^d).
	\end{equation}
	For this particular choice of $\mu\in M_1(F)$ we obtain the following:
	\begin{align}
		\label{here}
		&\mathbb{E}^a\big(| \Lambda_* + V_*|^2\big)\\\notag
		&\ge \mathbb{E}^a\bigg(\bigg[\sum_{x\in\Z^d}
			\frac{g(X_M\,,x)}{g(a\,,x)}
			\left\{p^{-{\rm d}(x,X_M)}
			\1_{\{x\leftrightarrow  X_M\}}
			+\1_{\{x\not\leftrightarrow  X_M\}}\right\}\mu(x)\bigg]^2\,\bigg|
			G\bigg)\cdot\mathbb{P}^a(G)\\\notag
		&\ge \bigg[\mathbb{E}^a\bigg(\sum_{x\in\Z^d}
			\frac{g(X_M\,,x)}{g(a\,,x)}
			\left\{p^{-{\rm d}(x,X_M)}
			\1_{\{x\leftrightarrow  X_M\}}
			+\1_{\{x\not\leftrightarrow  X_M\}}\right\}\mu(x)\,\bigg|
			G\bigg)\bigg]^2\cdot\mathbb{P}^a(G)\\\notag
		&= \left[ I(\mu\,;a)+I_p(\mu\,;a)\right]^2
			\cdot\mathbb{P}^a(G),
	\end{align}
	$\mathbb{P}^a$-a.s., thanks to the Cauchy--Schwarz inequality
	and out special choice of $\mu$ in \eqref{mu}. [The conditional expectation
	is well defined since $\mathbb{P}^a(G)>0$.]
	
	Recall that the forest representation of $\Z^d$ in \S\ref{sec:forest} identifies 
	$\rho\in\cV_k\cap\Z^d$
	with a certain finite subset of the real line. With this in mind,
	we see that  $\{\Lambda_{m,\rho}\}_{m\ge 0,
	\rho\in\cV_k\cap\Z^d}$ is a 2-parameter martingale
	under the probability
	measure $\mathbb{P}^a$, in the sense of
	Cairoli \cite{Cai1977}, with respect to the 2-parameter filtration 
	\begin{equation}
		\left\{\mathcal{X}_m\vee
		\mathcal{P}_\rho\right\}_{m\ge 0,\rho\in\cV_k\cap \Z^d}.
	\end{equation}
	The latter filtration satisfies the commutation hypothesis (F4) of Cairoli \cite{Cai1977}
	because $\mathcal{X}$
	and $\mathcal{P}$ are independent; 
	see Khoshnevisan \cite[\S3.4, p.\ 35]{Kh2002} for a more modern account.
	Therefore,
	Cairoli's maximal inequality \cite[Corollary 3.5.1, p.\ 37]{Kh2002} implies that
	$\mathbb{E}^a(\Lambda^2_*)\le 16\sup_{m,\rho}\mathbb{E}^a(\Lambda_{m,\rho}^2).$
	This and Jensen's inequality together imply that
	\begin{equation}
		\mathbb{E}^a(\Lambda^2_*)\le 16\mathbb{E}^a(J_{\mu}^2).
	\end{equation}
	
	Similarly, we can prove that
	$\mathbb{E}^a(V^2_*)\le 16\mathbb{E}^a(J_{\mu}^2).$
	Therefore, we may combine these remarks with \eqref{EJ^2} 
	in the following manner:
	\begin{equation}\label{BLIP}
		\mathbb{E}^a\left(|\Lambda_*+V_*|^2\right) \le 64\mathbb{E}^a(J_{\mu}^2)
		\le 128\left\{ I(\mu\,;a) + I_p(\mu\,;a)\right\}.
	\end{equation}
	Because of the above bound and \eqref{here},
	and since $\mathbb{P}^a(G)>0$ [see \eqref{assume}], it follows that
	$I(\mu\,;a)+I_p(\mu\,;a)$ is 
	strictly positive and finite. Therefore, we may resolve \eqref{here} 
	using \eqref{BLIP} in order to obtain the inequality
	\begin{equation}
		\mathbb{P}^a(G) \le \frac{128}{I(\mu\,;a)+I_p(\mu\,;a)} \le 128 c_p(F\,;a).
	\end{equation}
	This completes our proof. 
\end{proof}

\section{Macroscopic Minkowski dimension}\label{sec:Dim_M}
Let us recall \cite{BT1989,BT1992} that the \emph{macroscopic
upper Minkowski dimension} of a set $A\subset\Z^d$ is defined as\footnote{%
Barlow and Taylor write ${\rm dim}_{\rm UM}$ in place of our
$\overline{{\rm Dim}}_{\rm M}$.}
\begin{equation}
	\overline{\text{Dim}}_{\rm M}(A) :=
	\limsup_{n\to\infty} n^{-1}\log_2(\card\left(A\cap\cV_n\right)),
\end{equation}
where $\log_2$ is the usual logarithm in base two.

In analogy with the usual [small-scale] theory of the  dimensions, one always has
\begin{equation}
	\Dim (A) \le \overline{\text{Dim}}_{\rm M}(A),
\end{equation}
for all sets $A\subseteq\Z^d$; see Barlow and Taylor
\cite{BT1992}. The Minkowski dimension is perhaps the most commonly used
notion of large-scale dimension, in some form or another, in part because it is
easy to understand and in many cases compute. 

In the context of random walks, we have the following
elegant formula for the Minkowski dimension of the range of a transient
random walk on $\Z^d$.

\begin{theorem}\label{th:Dim_M}
	Let $X$ denote a transient random walk on $\Z^d$, with
	Green's function $g$, as before. Then, with probability one,
	\begin{equation}\label{eq:Dim_M}
		\overline{\text{\rm Dim}}_{\rm M}(\mathcal{R}_X)=\gamma_c,
	\end{equation}
	where
	\begin{equation}\label{eq:gamma_c}
		\gamma_c := \inf\left\{
		\gamma\in(0\,,d):\ \sum_{x\in\Z^d\setminus\{0\}} \frac{g(0\,,x)}{
		\|x\|^\gamma}<\infty\right\},
	\end{equation}
	where $\inf\varnothing:=d$.
\end{theorem}

The proof hinges on the analysis of the \emph{0-potential measure}
\begin{equation}
	U(A) :=  \sum_{n=0}^\infty P^0\{X_n\in A\}=\sum_{x\in A} g(0\,,x),
\end{equation}
defined for all $A\subset\R^d$. Because $X$ is transient, the set function 
$U$ is a Radon Borel measure on $\R^d$. 
Since $g(x\,,y) = g(0\,,y-x)$ for all $x,y\in\Z^d$, it follows that for all $A\subset\R^d$,
\begin{equation}\label{R:U}
	E^0\left[ \card\left( \mathcal{R}_X\cap A\right)\right]
	=\sum_{x\in A} P^0\left\{ X_k=x\text{ for some $k\ge 0$}\right\}
	=\frac{U(A)}{g(0\,,0)},
\end{equation}
thanks to a combination of Tonelli's theorem and \eqref{HitPoints}.

The following simple argument proves immediately the first half of Theorem \ref{th:Dim_M}.

\begin{proof}[Proof of Theorem \ref{th:Dim_M}: Upper bound]
	We first prove that,  with probability one,
	\begin{equation}\label{eq:Dim_M:UB}
		\overline{\text{\rm Dim}}_{\rm M}(\mathcal{R}_X) \le 
		\gamma_c.
	\end{equation}
	The more involved converse bound
	will be proved later. 
	
	Chebyshev's inequality and \eqref{R:U}
	together show that for all real numbers
	$\gamma>0$ and integers $k\ge 1$,
	\begin{equation}\label{Boo1}
		P^0\left\{ \card(\mathcal{R}_X\cap \cS_k) \ge 2^{k\gamma}\right\}
		\le \frac{2^{-k\gamma}U(\cS_k)}{g(0\,,0)}.
	\end{equation}
	Because $g(x\,,y)\le g(0\,,0)<\infty$ for all $x,y\in\Z^d$---see
	\eqref{HitPoints}---%
	there exists a finite constant $b$ such that $U(\cS_k)\le b 2^{kd}$
	for all $k\ge 1$. Therefore, the sum over $k$ of the right-hand side
	of \eqref{Boo1} is always finite when $\gamma>d$.
	If $\gamma\in(0\,,d)$ is such that the right-hand side of
	\eqref{Boo1} forms a summable series [indexed 
	by $k$], then so does the left-hand side. 
	The Borel--Cantelli lemma shows
	that for any such value of $\gamma$ 
	the random variable $L_{\gamma}$ defined by 
	\begin{equation}
		L_{\gamma} := \sup_{k \in \N} \left\{ 
		\frac{\card(\mathcal{R}_X\cap\cS_k)}{2^{k\gamma}}\right\}.
	\end{equation}
	is a.s.\ finite. In particular,
	\begin{equation}\label{eq:cb}
		\card(\mathcal{R}_X\cap\cV_k) \le \card(\cV_0)+
		L_\gamma\sum_{j=1}^k 2^{j\gamma} \le  
		2^{\gamma} (L_{\gamma}\vee 4^d)2^{k\gamma},
	\end{equation}
	for all $k\ge 1$. This proves that
	\begin{equation}
		\limsup_{n\to\infty} n^{-1} \log_2(\card\left(\mathcal{R}_X
		\cap\cV_n\right)) \le\gamma\qquad\text{a.s.,}
	\end{equation}
	whence
	$\overline{\text{Dim}}_{\rm M}(\mathcal{R}_X)\le\gamma$ a.s.\ for such
	a $\gamma$. Optimize over all such $\gamma$'s in order to see that
	\begin{equation}\label{eq:Dim:UM}
		\overline{\text{Dim}}_{\rm M}(\mathcal{R}_X)\le\inf\Big\{
		\gamma\in(0\,,d):\ \sum_{k=1}^\infty 2^{-k\gamma}U(\cS_k)<\infty\Big\},
	\end{equation}
	where $\inf\varnothing:=d$. To finish, note that
	if $x\in\cS_k$ then $\|x\| \ge 2^{k-1}$, whence
	\begin{equation}\label{U>g}\begin{split}
		\sum_{k=1}^\infty 2^{-k\gamma} U(\cS_k) &= 
			\sum_{k=1}^\infty 2^{-k\gamma}\sum_{x\in \cS_k} g(0\,,x)\\
		&\ge 2^{-\gamma}\sum_{k=1}^\infty \|x\|^{-\gamma}
			\sum_{x\in \cS_k} g(0\,,x)\\
		&= 2^{-\gamma} \sum_{x\in\Z^d\setminus\cV_0}
			\frac{g(0\,,x)}{\|x\|^\gamma}.
	\end{split}\end{equation}
	This and \eqref{eq:Dim:UM} together imply \eqref{eq:Dim_M:UB}.
\end{proof}

For the remaining, more challenging, direction of Theorem \ref{th:Dim_M}
we need to know that the measure $U$ is
\emph{volume-doubling}. That is the gist of the following result.

\begin{proposition}\label{pr:VolDouble}
	$U(\cV_{n+1}) \le 4^dU(\cV_n)$ for all $n\ge 0$.
\end{proposition}

This is a volume-doubling result because $\cV_n=2\cV_{n-1}$.
See Khoshnevisan and Xiao \cite{KX2003} 
for a corresponding result for L\'evy processes on $\R^d$.

\begin{proof}
	The proposition holds trivially when $n=0$. Therefore, we will concentrate
	on the case $n\ge 1$ from now on.
	
	We begin with a familiar series of random-walk computations.
	Choose and fix  an integer  $n\ge 1$ and some $x\in\Z^d$. Then, we
	apply the strong Markov property at 
	$\tau:=\inf\{k\ge0:\, X_k\in x+\cV_{n-1}\}$ $[\inf\varnothing:=+\infty]$
	in order to see that
	\begin{equation}
		U(x+\cV_{n-1}) = E^0\left[U\left(-X_\tau+x+\cV_{n-1}\right)
		;\, \tau<\infty\right].
	\end{equation}
	Since $-X_\tau+x\in-\cV_{n-1}$ $P^0$-a.s.\ on $\{\tau<\infty\}$, and
	$-\cV_{n-1}+\cV_{n-1}=\cV_n$, this readily yields the ``shifted-ball
	inequality,''
	\begin{equation}\label{SBI}
		\sup_{x\in\Z^d} U(x+\cV_{n-1})\le U(\cV_n)\qquad\text{for
		all $n\ge 1$}.
	\end{equation}
	Eq.\ \eqref{SBI} becomes obvious if ``$\sup_{x\in\Z^d}$'' were replaced by
	``$\sup_{x\in\cV_{n-1}}$.'' The strong Markov property of $X$ was 
	needed in order to establish this  improvement.
		
	Armed with \eqref{SBI} we proceed in a standard way:
	We can always find $4^d$ integer 
	points $x_1,\ldots,x_{4^d}\in\Z^d$ such that
	\begin{equation}
		\cV_{n+1} = \bigcup_{j=1}^{4^d} \left( x_j + \cV_{n-1}\right),
	\end{equation}
	for all $n\ge 1$, where the union is a disjoint one. Thus,
	\begin{equation}
		U(\cV_{n+1}) = \sum_{j=1}^{4^d} U(x_j+\cV_{n-1})
		\le 4^d\sup_{x\in\Z^d}U(x+\cV_{n-1}).
	\end{equation}
	The proposition follows from this and \eqref{SBI}.
\end{proof}

Next we rewrite $\gamma_c$---see \eqref{eq:gamma_c}---in a
slightly different form. We will be ready to complete the 
proof of Theorem \ref{th:Dim_M} once that task is done.

\begin{proposition}\label{pr:gamma_c}
	$\gamma_c = \limsup_{n\to\infty} n^{-1}\log_2 U(\cV_n)$.
\end{proposition}

\begin{proof}
	If $x\in\cS_k$, then $\|x\|\le d^{1/2} 2^k$. Therefore,
	\begin{equation}\label{U<g}
		\sum_{k=1}^\infty 2^{-k\gamma} U(\cS_k) 
		\le  d^{\gamma/2}\sum_{x\in\Z^d\setminus\cV_0}
		\frac{g(0\,,x)}{\|x\|^\gamma}.
	\end{equation}
	Therefore, we can infer from \eqref{U>g} that
	\begin{equation}\label{eq:gamma_c:1}
		\gamma_c = \inf \bigg\{
		\gamma\in(0\,,d):\ \sum_{k=1}^\infty
		2^{-k\gamma} U(\cS_k)<\infty \bigg\}.
	\end{equation}
	
	We apply \eqref{eq:gamma_c:1} to rewrite $\gamma_c$ once more time:
	If $\gamma>\gamma_c$, then $U(\cS_k)=o(2^{k\gamma})$ as $k\to\infty$.
	If on the other hand $\gamma\in(0\,,\gamma_c)$, then 
	we can argue by contraposition to see that, for every fixed $\epsilon>0$,
	$U(\cS_k)> 2^{k(\gamma-\epsilon)}$ for infinitely-many integers $k$. This means
	that
	\begin{equation}\label{eq:gamma_c:2}
		\gamma_c = \limsup_{n\to\infty} n^{-1}\log_2 U(\cS_n).
	\end{equation}
	Now we prove the proposition.
	
	The assertion of the proposition is that $\gamma_c=\gamma_c'$, where
	\begin{equation}
		\gamma_c' := \limsup_{n\to\infty} n^{-1}\log_2 U(\cV_n).
	\end{equation}
	On one hand,  \eqref{eq:gamma_c:2} implies that $\gamma_c\le\gamma_c'.$
	If, on the other hand, $\vartheta>\gamma_c$ is an arbitrary finite number, then
	there exists a finite constant $L_\vartheta$ such that $U(\cS_k)\le L_\vartheta 2^{k\vartheta}$
	for all integers $k\ge 1$. In particular,
	\begin{equation}
		U(\cV_n) \le \#(\cV_0)+L_\vartheta \sum_{k=1}^n 2^{k\vartheta}
		\qquad\text{for all $n\ge 1$},
	\end{equation}
	whence follows that $U(\cV_n)=O(2^{n\vartheta})$ as $n\to\infty$.
	Since this is true for all $\vartheta>\gamma_c$, it follows that
	$\gamma_c'\le\gamma_c$, as was promised.
\end{proof}

\begin{proof}[Proof of Theorem \ref{th:Dim_M}: Lower bound]
	It remains to prove that
	\begin{equation}
		\overline{\text{\rm Dim}}_{\rm M}(\mathcal{R}_X) \ge \gamma_c
		\qquad\text{a.s.,}
	\end{equation}
	where $\gamma_c$ was defined in \eqref{eq:gamma_c}. If $\gamma_c=0$, then
	we are done. Therefore, from now on we assume without loss of generality that
	\begin{equation}\label{lolo}
		\gamma_c>0.
	\end{equation}
	
	Define 
	\begin{equation}
		\tau(x) := \inf\left\{ n\ge 0:\, X_n=x\right\},
	\end{equation}
	for all $x\in\Z^d$ [$\inf\varnothing:=+\infty$].
	Since $\card(\mathcal{R}_X\cap A)=\#\{x\in A:\, \tau(x)<\infty\}$,
	Tonelli's theorem implies that
	\begin{equation}\label{E2}\begin{split}
		&E^0\left(\left|\card\left( \mathcal{R}_X\cap\cV_n\right)\right|^2\right)\\
		&=E^0\left[\card\left(\mathcal{R}_X\cap\cV_n\right)\right] +
			2\mathop{\sum\sum}\limits_{\substack{x,y\in\cV_n\\x\neq y}}
			P^0\left\{ \tau(x)<\tau(y)<\infty\right\}\\
		&= E^0\left[\card\left(\mathcal{R}_X\cap\cV_n\right)\right]+ 2
			\mathop{\sum\sum}\limits_{\substack{x,y\in\cV_n\\x\neq y}}
			P^0\{ \tau(x)<\infty\}P^x\{\tau(y)<\infty\},
	\end{split}\end{equation}
	thanks to the strong Markov property. If $x\in\cV_n$ and $n\ge 1$ are
	held fixed, then
	\begin{equation}
		\sum_{y\in\cV_n\setminus\{x\}} P^x\{\tau(y)<\infty\}
		= \frac{U(\cV_n-x)}{g(0\,,0)}\le \frac{U(\cV_{n+1})}{g(0\,,0)},
	\end{equation}
	since $\cV_n-x\subset\cV_n-\cV_n=\cV_{n+1}$. Therefore, \eqref{E2}
	implies that
	\begin{equation}\label{E3}\begin{split}
		&E^0\left(\left|\card\left( \mathcal{R}_X\cap\cV_n\right)\right|^2\right)
			\le E^0\left[\card\left(\mathcal{R}_X\cap\cV_n\right)\right]
			+ \frac{2U(\cV_n)U(\cV_{n+1})}{[g(0\,,0)]^2}\\
		&\qquad\le E^0\left[\card\left(\mathcal{R}_X\cap\cV_n\right)\right]
			+ 2^{1+2d}\left\{E^0\left[\card
			\left(\mathcal{R}_X\cap\cV_n\right)\right]\right\}^2,
	\end{split}\end{equation}
	thanks to \eqref{R:U} and Proposition \ref{pr:VolDouble}. Because
	\begin{equation}
		\E[\card(\mathcal{R}_X\cap\cV_n)]= \frac{U(\cV_n)}{g(0\,,0)},
	\end{equation}
	Eq.\ \eqref{E3} and the
	Paley--Zygmund inequality together yield the following: For 
	infinitely many $n$,
	\begin{equation}\begin{split}
		P^0\left\{ \card\left( \mathcal{R}_X\cap\cV_n\right)
			> \frac{U(\cV_n)}{2g(0\,,0)} \right\}
		&\ge \frac{\left\{E^0\left[\card
			\left(\mathcal{R}_X\cap\cV_n\right)\right]\right\}^2
			}{4E^0\left(\left|\card\left( \mathcal{R}_X\cap\cV_n\right)\right|^2\right)}\\
		&\ge \frac{1}{4(1+2^{1+d})}\\
		&:= \varrho(d).
	\end{split}\end{equation}
	The last part holds since \eqref{lolo}
	implies that $\E[\card(\mathcal{R}_X\cap\cV_n)]\ge1$ for 
	infinitely-many integers $n>1$. The preceding displayed inequality
	and Proposition \ref{pr:gamma_c} together imply that
	$\overline{\text{\rm Dim}}_{\rm M}(\mathcal{R}_X) \ge\gamma_c$,
	with probability at least $\varrho(d)>0$. 
	Since $\overline{\text{\rm Dim}}_{\rm M}(\mathcal{R}_X) =
	\overline{\text{\rm Dim}}_{\rm M}(\mathcal{R}_X\cap \cV_N^c)$ for all $N\ge1$,
	an application of the Hewitt--Savage 0--1 law shows the desired result that
	$\overline{\text{\rm Dim}}_{\rm M}(\mathcal{R}_X)\ge\gamma_c$ almost surely.
\end{proof}

\section{Concluding remarks and open problems}\label{sec:conj}
Corollary \ref{cor:main} succeeds in yielding a
formula for $\Dim(\mathcal{R}_X\cap F)$ for every recurrent set $F\subset\Z^d$,
though it is difficult to work with that formula. We do not expect a simple formula for $\Dim(\mathcal{R}_X\cap F)$ 
when $F$ is a general recurrent set in $\Z^d$.
In fact, it is not even easy to decide whether or not a general set $F$ is recurrent,
as we have seen already. However, one can hope for simpler descriptions of
$\Dim(\mathcal{R}_X\cap F)$ when $F=\Z^d$.
 In this section we conclude with a series of remarks, problems, and conjectures
 that have these comments in mind.

Question \eqref{BT:question} was in part motivated by its ``local variation,'' which
had  been open since the mid-to-late 1960's \cite{Pruitt},
and possibly earlier. Namely, let
$\{y(t)\}_{t\ge0}$ be a L\'evy process in $\R^d$. 
The local version of \eqref{BT:question} asks, ``what
is the ordinary Hausdorff dimension $\dim_{\rm H}$ of the range 
$y(\R_+):=\cup_{t\ge0}\{y(t)\}$?''
This question was answered several years later 
by Khoshnevisan, Xiao, and Zhong \cite[Corollary 1.8]{KXZ2003}, who showed
among other things that $\dim_{\rm H}(y(\R_+))$ is a.s.\ equal 
to an index that was introduced earlier in Pruitt \cite{Pruitt}
as part of the solution to the very same question. Under a quite mild
regularity condition, it has been shown that the general formula for
$\dim_{\rm H}(y(\R_+))$ reduces to the following \cite[(1.19) of Theorem 
1.5]{KX2009}:
\begin{equation}\label{dimh}
	\dim_{\rm H}(y(\R_+)) = \sup\left\{\gamma\in(0\,,d):\
	\int_{\R^d} \frac{u(x)}{\|x\|^\gamma}\,\d x<\infty\right\}
	\quad\text{a.s.},
\end{equation}
where $u$ denotes the 1-potential kernel of $y$. 
Khoshnevisan and Xiao \cite[\textcolor{red}{(1.4)}]{KX2008}
find an alternative Fourier-analytic formula. 

If we proceed purely
by analogy, then we might guess from \eqref{dimh} and \eqref{eq:gamma_c}
the following formula for the macroscopic
Hausdorff dimension of the range $\mathcal{R}_X$ of our random walk $X$
on $\Z^d$: 
\begin{equation}\label{conj}
	\Dim(\mathcal{R}_X) = \gamma_c
	\qquad\text{a.s.}
\end{equation}
In principle, we ought to be able to decide whether or not \eqref{conj}
is correct, based solely on Corollary \ref{cor:main}. But we do not know how
to do that at this time mainly because it is quite difficult to compute
$p_c(\Z^d\,;0)$ when $X$ is a general transient
random walk. Instead, we are able to only offer
\begin{CONJ}	\label{OP1}
	$\Dim(\mathcal{R}_X)=\gamma_c$ a.s.\
	for every transient random walk on $\Z^d$,
	where $\gamma_c$ was defined in \eqref{eq:gamma_c}.
\end{CONJ}

Because of Theorem \ref{th:Dim_M},
Conjecture \ref{OP1} is equivalent to the assertion that
$\Dim(\mathcal{R}_X)=\overline{{\rm Dim}}_{\rm M}(\mathcal{R}_X)$ a.s.
It is known that the ordinary [microscopic] 
Hausdorff dimension of the range of a L\'evy
process is always equal to its ordinary  [microscopic]  
\emph{lower} Minkowski dimension,
and not always the upper Minkowski dimension. If
Conjecture \ref{OP1} were correct, then it would suggest that
large-scale dimension theory of random walks is somewhat different 
from its small-scale counterpart. Our next Problem is an attempt to understand
this difference better.

Barlow and Taylor \cite{BT1989,BT1992} have introduced two other notions of
macroscopic dimension that are  related to our present interests. 
Namely, they define the [macroscopic] \emph{lower
Minkowski dimension} of $A\subset\Z^d$ 
and the \emph{lower Hausdorff dimension} of $A\subset\Z^d$ respectively as\footnote{%
Barlow and Taylor write ${\rm dim}_{\rm LM}$ and $\dim_{\rm L}$
in place of our $\underline{{\rm Dim}}_{\rm M}$ and $\underline{{\rm Dim}}_{\rm H}$.}
\begin{equation}\begin{split}
	\underline{{\rm Dim}}_{\rm M} (A) &:= \liminf_{n\to\infty}
		n^{-1}\log \card(A\cap\cV_n),\\
	\underline{{\rm Dim}}_{\rm H} (A) &:= \inf\left\{\alpha>0:
		\lim_{k\to\infty}\mathcal{N}_\alpha(A\,,\cS_k)=0\right\}.
\end{split}\end{equation}
One has $\underline{\rm Dim}_{\rm H}(A)\le \Dim(A)$ and 
$\underline{{\rm Dim}}_{\rm M}(A)\le \overline{\rm Dim}_{\rm M}(A)$ 
for all $A\subseteq\Z^d$.

It is easy to obtain a nontrivial upper bound for the lower Minkowski dimension
of $\mathcal{R}_X$, valid for every transient random walk $X$ on $\Z^d$.
Namely, by \eqref{R:U} and Fatou's lemma,
\begin{equation}
	\E\left[\liminf_{n\to\infty} 
	2^{-n\gamma}\card(\mathcal{R}_X\cap\cV_n) \right]
	\le \liminf_{n\to\infty} 2^{-n\gamma} U(\cV_n),
\end{equation}
for every $\gamma\in[0\,,\infty)$.
From this we readily can deduce that
\begin{equation}\label{ubub}
	\underline{\rm Dim}_{\rm M}(\mathcal{R}_X) \le \liminf_{n\to\infty}
	n^{-1}\log U(\cV_n)\qquad\text{a.s.}
\end{equation}
We believe that this is a sharp bound, and thus propose the following.
\begin{CONJ}\label{OP2}
	With probability one,
	\begin{equation}
		\underline{{\rm Dim}}_{\rm H}(\mathcal{R}_X)
		=\underline{\rm Dim}_{\rm M}(\mathcal{R}_X) = \liminf_{n\to\infty}
		n^{-1}\log U(\cV_n).
	\end{equation}
\end{CONJ}
Admittedly, we have not tried very hard to prove this, but it seems to be a
natural statement. There are two other good reasons for our interest in 
Conjecture \ref{OP2}. First of all, it suggests that, as far as random walks
and their analogous L\'evy processes are
concerned, the more natural choice of ``macroscopic Hausdorff
dimension'' is $\underline{\rm Dim}_{\rm H}$ and not $\Dim$,
in contrast with the proposition of \cite{BT1989,BT1992}. Also, if
Conjecture \ref{OP2} were true, then together with Theorem \ref{th:Dim_M}
and Proposition \ref{pr:gamma_c}
it would imply that regardless of whether or not Conjecture \ref{OP1}
is true, $\Dim(\mathcal{R}_X)$ always lies in the nonrandom
interval
$[ \liminf_{n\to\infty} n^{-1}\log U(\cV_n)\,,
\limsup_{n\to\infty} n^{-1}\log U(\cV_n)].$
The extrema of this interval are typically not hard to compute; therefore,
we at least will have easy-to-compute bounds for $\Dim(\mathcal{R}_X)$.

Let us state a third conjecture that is motivated also by Conjecture \ref{OP1}. 

Choose and fix an arbitrary integer $N\ge 1$, and define $X^{(1)},\ldots,X^{(N)}$
to be $N$ independent copies of a symmetric,
transient random walk $X$ on $\Z^d$
whose Green's function satisfies the Barlow--Taylor condition \eqref{cond:BT}
for some $\alpha\in(0\,,2]$. We can define an $N$-parameter \emph{additive
random walk} $\mathfrak{X}:=\{\mathfrak{X}(\bm{n})\}_{\bm{n}\in\Z^N_+}$
as follows \cite[Ch.\ 4]{Kh2002}: For every $\bm{n}:=(n_1\,,\ldots,n_N)\in\Z^N_+$,
\begin{equation}
	\mathfrak{X}(\bm{n}) := X^{(1)}_{n_1} + \cdots + X^{(N)}_{n_N}.
\end{equation}
 Let $\mathcal{R}_{\mathfrak{X}}:=\cup_{\bm{n}\in\Z^N_+}
 \{\mathfrak{X}(\bm{n})\}$ denote the \emph{range} of the random
 field $\mathfrak{X}$.
 
 \begin{CONJ}\label{OP3}
 	Suppose  $d>\alpha N$ and $N>1$. Then for all nonrandom
	$A\subset\Z^d$:
	\begin{enumerate}
		\item If $\Dim(A)>d-\alpha N$, then $\mathcal{R}_{\mathfrak{X}}\cap A$
			is a.s.\ unbounded; \underline{and}
		\item If $\Dim(A)<d-\alpha N$, then $\mathcal{R}_{\mathfrak{X}}\cap A$
			is a.s.\ bounded.
	\end{enumerate}
\end{CONJ} 
Proposition \ref{pr:BT}  implies that Conjecture \ref{OP3}
is correct if $N$ were replaced by $1$; the case $N>1$ has eluded our many
attempts at solving this problem.

It is possible to adapt the arguments
of \cite{KXZ2003} in order to derive 
Conjecture \ref{OP1} from
Conjecture \ref{OP3}. We skip the details of that argument. Instead,
we conclude with two problems about the ``continuous version''
of Corollary \ref{cor:main}, which we recall,
contained our Hausdorff dimension
formula for the range of a walk.

\begin{OP}
	Let $\{y(t)\}_{t\ge0}$ be a transient,
	but otherwise general, L\'evy process on $\R^d$
	whose characteristic exponent is $\Psi$, normalized as
	\begin{equation}
		\E\left(\e^{iz\cdot y(t)}\right) =\e^{-t\Psi(z)}
		\qquad\text{for all $z\in\R^d$
		and $t\ge 0$,}
	\end{equation}
	to be concrete.
	Is there a formula for the a.s.-constant
	quantity $\Dim(y(\R_+))$ that is solely in terms
	of $\Psi$?
\end{OP}

Before we state our last question let
us define the \emph{upper Minkowski dimension} of a set $A\subseteq\R^d$
as follows: Define $A'$ to be the 
union of all dyadic cubes $Q\in\D_0$ of sidelength 
one that intersect $A$.

\begin{definition}
	The \emph{macroscopic upper Minkowski dimension} $\overline{\rm Dim}_{\rm M}(A)$
	is defined, via the Barlow--Taylor upper Minkowski dimension, as follows:
	\begin{equation}
		\overline{\rm Dim}_{\rm M}(A) := \overline{\rm Dim}_{\rm M}(A')\qquad
		\text{for all $A\subset\R^d$}.
	\end{equation}
\end{definition}
The same proof that worked for $A\subseteq\Z^d$ works to show that
$\Dim(A)\le\overline{\rm Dim}_{\rm M}(A)$ for all $A\subset\R^d$.

Although we have not checked all of the details, we believe that
the method of proof of Theorem \ref{th:Dim_M} can be adapted
to the continuous setting in order to produce
\begin{equation}\label{L:M}
	\overline{\rm Dim}_{\rm M}(y(\R_+)) = 
	\limsup_{n\to\infty} n^{-1}\log \mathbb{U}(\cV_n)
	\qquad\text{a.s.},
\end{equation}
where $\mathbb{U}(A) := \int_0^\infty \P\{y(s)\in A\}\,\d s$
for all Borel sets $A\subset\R^d$.

\begin{OP}
	Can one write an expression for
	$\overline{\rm Dim}_{\rm M}(y(\R_+))$ solely in terms of $\Psi$?
\end{OP}

Finally, we mention that Conjecture \ref{OP3} is likely to have
a L\'evy process version wherein the role of the $X^{(i)}$s are replaced
by that of isotropic $\alpha$-stable L\'evy processes. 
We leave the statement [and perhaps also 
a proof!] to the interested reader.

\section*{Acknowledgments}

Alex D. Ramos  thanks CNPq-Brazil and CAPES, process BEX 2176/13-0,
for their financial support, and the Mathematics Department of 
University of Utah for their  hospitality during his visit.

\spacing{.8}\bigskip

\noindent\textbf{Nicos Georgiou}\\
\noindent Department of Mathematics, University of Sussex, 
	Sussex House, Brighton, BN1 9RH, United Kingdom\\
\noindent\textcolor{purple}{\texttt{N.Georgiou@sussex.ac.uk}}\\

\noindent\textbf{Davar Khoshnevisan} \&\ \textbf{Kunwoo Kim}\\
\noindent Department of Mathematics, University of Utah,
		Salt Lake City, UT 84112-0090 \\
\noindent\textcolor{purple}{\texttt{davar@math.utah.edu}}
	\&\ \textcolor{purple}{\texttt{kkim@math.utah.edu}}\\
	
\noindent\textbf{Alex D. Ramos}\\
	Federal University of Pernambuco,
	Department of Statistics/CCEN,
	Av. Prof. Luiz Freire, s/n,
	Cidade Universit\'aria,
	Recife/Pe, 50740-540, Brazil\\
\noindent\textcolor{purple}{\texttt{alex@de.ufpe.br}}

\end{document}